\documentclass [11pt]{amsart}
\usepackage{comment}
\usepackage{stmaryrd}
\usepackage{amssymb} \usepackage{amsfonts} \usepackage{amsmath}
\usepackage{amsthm} \usepackage{epsfig} 
\usepackage[cmtip,arrow]{xy}
\usepackage{amsmath,amssymb,enumerate,pb-diagram,pb-xy}
\usepackage{color}

\usepackage[applemac]{inputenc}
\numberwithin{equation}{section}
\input xy
\xyoption{all}
\usepackage{caption}

\renewcommand{\phi}{\varphi}

\newcommand{\bqn}{\begin{equation*}}
\newcommand{\eqn}{\end{equation*}}
\newcommand{\bq}{\begin{equation}}
\newcommand{\eq}{\end{equation}}
\newcommand{\ba}{\begin{aligned}}
\newcommand{\ea}{\end{aligned}}
\newcommand{\be}{\begin{enumerate}}
\newcommand{\ee}{\end{enumerate}}

\newtheorem{lemma}{Lemma}[section]

\newtheorem{thm}[lemma]{Theorem}
\newtheorem{prop}[lemma]{Proposition}
\newtheorem{cor}[lemma]{Corollary}

\theoremstyle{definition}
\newtheorem{defn}[lemma]{Definition}

\newtheorem{example}[lemma]{Example}
\newtheorem{rem}[lemma]{Remark}

\newcommand{\R} {\ensuremath {\mathbb{R}}}

\newcommand{\eu}{\ensuremath{\mathrm{eu}}}

\newcommand{\omeo}{\ensuremath{{\rm Homeo}^+(S^1)}}

\newcommand{\G}{\ensuremath {\Gamma}}

\newcommand{\Z}{\mathbb Z}

\author{Michelle Bucher}
\author{Roberto Frigerio}
\author{Tobias Hartnick}
\date{May 2016}

\title{A note on semi-conjugacy for circle actions}
\begin{document}
\maketitle
\begin{abstract} We define a notion of semi-conjugacy 
between orientation-preserving actions of a group on the circle, 
which for fixed point free actions coincides 
with a classical definition of Ghys.
We then show that two circle actions are semi-conjugate if and only if they have the same bounded Euler class. This settles some existing confusion present in the literature.
\end{abstract}
\section{Introduction}

A fundamental problem in one-dimensional dynamics is the classification of group actions on the circle. More precisely, denote by $\omeo$ the group of orientation-preserving homeomorphisms of the circle. Given a group $\G$, 
we will refer to a homomorphism $\rho: \G \to \omeo$ as a \emph{circle action}. One would like to associate to every circle action of $\G$ a family of invariants which classify the action up to a suitable equivalence relation, 
ideally up to conjugacy. For the case of a single transformation acting minimally on the circle, this problem was solved by Poincar\'e around the end of the 19th century, using his theory of rotation number \cite{Poi1, Poi2}.

In \cite{Ghys0,Ghys2} \'Etienne Ghys introduced and studied a far reaching generalization of the rotation number, the bounded Euler class of a circle action. For \emph{minimal} actions, i.e. actions for which every orbit is dense, 
he thereby achieved a complete classification result:
\begin{thm}[{\cite[Theorem 6.5]{Ghys2}}]\label{GhysMinimal}
Let $\rho_1, \rho_2: \Gamma \to \omeo$ be minimal circle actions. Then $\rho_1$ and $\rho_2$ are conjugate if and only if they have the same bounded Euler class.
\end{thm}
The bounded Euler class is thus a complete conjugacy-invariant for minimal actions. For non-minimal actions, this result is not true. Instead, non-minimal actions sharing the same bounded Euler class only satisfy
 a weaker equivalence relation. In \cite{Ghys0} Ghys introduced the notion of \emph{semi-conjugacy} between circle actions, which generalizes the notion of conjugacy. With this notion he proved:
\begin{thm}[{\cite[Theorem A1]{Ghys0}}]\label{IntroMain}
Two circle actions $\rho_1, \rho_2: \Gamma \to \omeo$ are semi-conjugate if and only if they have the same bounded Euler class.
\end{thm}
The bounded Euler class which appears in Theorem \ref{GhysMinimal} and Theorem \ref{IntroMain} is an invariant with values in the second bounded cohomology $H^2_b(\Gamma; \Z)$ of $\Gamma$ with $\Z$-coefficients. 
The theory of Ghys developed in \cite{Ghys0, Ghys2} goes far beyond Theorem \ref{IntroMain}. Namely, not only does it parametrize semi-conjugacy classes of circle actions by classes in $H^2_b(\Gamma; \Z)$, but it
 also characterizes exactly which classes in $H^2_b(\Gamma; \Z)$ can be realized by circle actions. This then provides a bijection between semi-conjugacy classes of circle actions and a certain explicit subset of
 $H^2_b(\Gamma; \Z)$. Although we will have nothing to say on this part of the theory in this note, let us at least state the main result:
\begin{thm}[{\cite[Theorem B]{Ghys0}}] Let $\Gamma$ be a discrete countable group and $\beta\in H^2_b(\Gamma,\Z)$. There exists a representation $\rho:\Gamma \rightarrow \omeo$ such that $\beta$ is
 the bounded Euler class of $\rho$ if and only if $\beta$ can be represented by a cocycle taking only the values $0$ and $1$. 
\end{thm}
Ghys' theory of the bounded Euler class has found applications in many different directions. Recently there has been a revived interest in Theorem \ref{IntroMain}, since it plays a fundamental role in the bounded cohomology approach to higher Teichm\"uller theory (\cite{BIW1, BIW2, BBHIW}).

The beginner in the field who is trying to understand the proof of Theorems \ref{GhysMinimal} and \ref{IntroMain} has to face several challenges which we try to address with this note. 

The first challenge is to understand the notion of bounded Euler class. Like ordinary cohomology, bounded cohomology can be defined either abstractly or through various concrete resolutions. 
In each concrete model the bounded Euler class is represented by a specific cocycle. For example, the proof of Ghys' Theorem makes use of two different incarnations of the bounded Euler class, 
namely the geometric description of the bounded Euler class associated with the $\omeo$-action on $S^1$, and the algebraic description in terms of translation numbers. Neither of these incarnations
 is particularly intuitive at first sight, 
and while it is well-known to the experts that  they represent the same cohomology class under a canonical isomorphism, this does not appear as obvious just by looking at the definitions. 

In our opinion, the most canonical way to define the bounded Euler class is to define it as the bounded lifting obstruction for the central extension corresponding to the universal covering of $\omeo$. 
This is the approach taken in the present note (see Definition \ref{defBddEulerclass}). We then carefully establish that the so-defined class can be represented over the circle by the well-known Euler cocycle 
(Corollary \ref{cor: e as cocycle on S1}) and can also be related to the translation number (Proposition \ref{TNumber}). This then shows in particular the equivalence of the two definitions used in the proof of Ghys' Theorem. 
Yet another characterization of the bounded Euler class in terms of the Sullivan cocycle over the double covering of the circle is given in the appendix. This description is crucial if one wants to extend the notion of
 bounded Euler class to higher dimensions and plays an important role in the study of the cohomology of ${\rm SL}_n(\R)$. It also allows us to give a different (and apparently new) characterization of circle actions with 
vanishing bounded Euler class, hence we include it here.

Once the notion of bounded Euler class is clarified, one needs to understand the notion of semi-conjugacy. 
Unfortunately, the original definition in \cite{Ghys0} suffered from a minor inaccuracy, which was corrected in later papers of the author. In the meantime, different authors had developed fixes of their own, leading to a plethora of alternative definitions. Right now the situation seems to be that all of
 these definitions are used simultaneously in the literature without much of a distinction. Several of the most used definitions can be shown to be equivalent and, more importantly, to satisfy Theorem \ref{IntroMain}. However there 
also appear several other definitions of semi-conjugacy in the literature, which are not equivalent and for which Theorem \ref{IntroMain} does not hold. The main goal of this article is to clarify the situation and to compare the 
different definitions. 

All definitions of semi-conjugacy start from the notion of a \emph{non-decreasing degree one map}, i.e. a map $\phi: S^1 \to S^1$ which admits a lift $\widetilde{\phi}: \R \to \R$ (called a \emph{good lift}) such that $\widetilde{\varphi}(x+1)=\widetilde{\varphi}(x)+1$ for every $x\in\R$ and $\widetilde{\varphi}$ is non-decreasing, i.e. $\widetilde{\varphi}(x)\leq \widetilde{\varphi}(y)$ whenever $x\leq y$. (In the body of this text, we will adopt the equivalent but more geometric point of view given in Definition \ref{Noddom}.)  

We emphasize that no continuity requirement is imposed in this definition, and hence the Brouwer-Hopf degree of $\phi$ may not be well-defined. Even if $\varphi$ happens to be continuous, 
it may still be constant and thus of Brouwer-Hopf degree $0$. In general, the Brouwer-Hopf degree of a \emph{continuous} non-decreasing degree one map is either $0$ or $1$ (and it is equal to zero if and only if the map is constant). 
We say that a non-decreasing degree one map $\varphi$ is \emph{upper/lower semi-continuous} if it admits a good lift with the corresponding property.

Now let $H :=\omeo$. We call a non-decreasing degree one map $\varphi: S^1 \to S^1$ a \emph{left-semi-conjugacy} from a circle action $\rho_1: \Gamma \to H$ to a circle action $\rho_2 : \Gamma \to H$ if 
 \[\rho_1(\gamma)\varphi=\varphi\rho_2(\gamma) \quad \text{for every} \quad \gamma\in \G.\]
We then call $\rho_1$ \emph{left-semi-conjugate}\footnote{In \cite{Ghys0} $\rho_1$ is simply called \emph{semi-conjugate} to $\rho_2$, but we would like to emphasize 
here the asymmetry in $\rho_1$ and $\rho_2$.} to $\rho_2$ and $\rho_2$ \emph{right-semi-conjugate} to $\rho_1$. 
\begin{thm}\label{ThmEquivalence} Let $\rho_1: \Gamma \to H$ and $\rho_2 : \Gamma \to H$ be circle actions of the same group $\Gamma$. Then the following are equivalent:
\begin{enumerate}[(i)]
\item$\rho_1$ is both left-semi-conjugate and right-semi-conjugate to $\rho_2$.
\item Either both $\rho_1(\Gamma)$ and $\rho_2(\Gamma)$ do not have a fixed point and $\rho_1$ is left-semi-conjugate to $\rho_2$, or $\rho_1(\Gamma)$ and $\rho_2(\Gamma)$ both have a fixed point.
\item There exist a left-semi-conjugacy $\phi$ from $\rho_1$ to $\rho_2$ and a $\rho_2(\Gamma)$-invariant subset $K \subset S^1$ such that $\phi|_K$ is injective.
\item There exist a left-semi-conjugacy $\phi$ from $\rho_1$ to $\rho_2$, lifts $\widetilde{\rho_1}(\gamma)$ and $\widetilde{\rho_2}(\gamma)$ for each $\gamma \in \Gamma$ and and a good lift  $\widetilde{\phi}$ of
 $\phi$ such that $\widetilde{\rho_1}(\gamma)\widetilde{\phi}(x) = \widetilde{\phi}(\widetilde{\rho_2}(\gamma)(x))$ for all $\gamma \in \Gamma$ and $x \in \R$.
\item $\rho_1$ and $\rho_2$ have the same bounded Euler class.
\end{enumerate}
All of these conditions remain equivalent if the left-semi-conjugacies in question are required to be either upper semi-continuous or lower semi-continuous.
\end{thm}
In this note we will define two circle actions $\rho_1$ and $\rho_2$ to be \emph{semi-conjugate} if they satisfy Condition (i) of the theorem (see Definition \ref{MainDef} below). The equivalence (i)$\Leftrightarrow$(v) 
is then exactly the content of Theorem \ref{IntroMain}. According to the theorem, each of the Conditions (ii)-(iv) could equally well be used as the definition of semi-conjugacy for Theorem \ref{IntroMain} to hold.

Definition (ii) is essentially Ghys' original definition (modulo the necessary correction in the case of fixed points). The case where both $\rho_1(\Gamma)$ and $\rho_2(\Gamma)$ have fixed points is actually equivalent
 to the vanishing of the bounded Euler class. One problem with Definition (ii) is that it is not obvious a priori whether it is an equivalence relation at all. From this point of view, Definition (i) is clearly preferable.
 The (re-)discovery of this ``symmetric'' definition by the second named author was one of our main motivations to write this note. (Later we learned from the referee that this definition already appeared in an old
 manuscript of Takamura \cite{Taka}, which however was never published.) Definition (iii) is due to the first-named author \cite{Bucherweb} and convenient to check in practice, since only one left-semi-conjugacy has to
 be constructed. Definition (iv) was kindly communicated to us by Maxime Wolff \cite{Wolff}.

\begin{rem} As was communicated to us by Ghys and is pretty clear from the proofs in \cite{Ghys0}\label{rem Ghys 1}, what was actually meant is a condition very close to Condition (iv) in Theorem \ref{ThmEquivalence}, which we state as Condition (vi) below. For this we observe that every circle action of $\Gamma$ gives rise
 to a central extension $\Gamma(\rho)$ of $\Gamma$ as follows. Denote by $\widetilde{H}$ the universal covering group of $H$, which is a central $\Z$-extension of $H$ and acts on the real line (see Subsection \ref{SubsecEuler1}). 
Then $\Gamma(\rho)$ is defined as the pullback
\[\begin{xy}\xymatrix{
\Gamma(\rho) \ar[r]^{\widetilde{\rho}} \ar[d] & \widetilde{H}\ar[d]\\
\Gamma \ar[r]^\rho & H.
}\end{xy}\]
We can now state Condition (vi) which is equivalent to (i)-(v) above:

\smallskip

\emph{(vi) There exist an isomorphism $\psi: \Gamma(\rho_1) \to \Gamma (\rho_2)$ and a good lift $\widetilde{\phi}$ of a non-decreasing degree one map such that for all $\gamma \in \Gamma(\rho_1)$ and $x \in \R$,}
\[
\widetilde{\rho_1}(\gamma)\widetilde{\phi}(x) = \widetilde{\phi}(\widetilde{\rho_2}(\psi(\gamma))(x)).
\]
It is obvious that it implies Condition (iv) of Theorem \ref{ThmEquivalence}. Condition (vi) has however the slight disadvantage that it requires the corresponding (unbounded) Euler classes to be equal, which is equivalent to the isomorphism between the two central extensions of $\Gamma$. We will  point out in Remark \ref{rem Ghys 2} how Condition (vi) immediately follows from Condition (i) of Theorem \ref{ThmEquivalence} based on the proof of Part (i) of Theorem \ref{GhysExplicit}.
\end{rem}

Having stated a number of equivalent definitions of semi-conjugacy, let us now point out a number of definitions we found in the literature, which are \emph{not} equivalent to the definitions above.
 For a more detailed discussion including various concrete counterexamples see Remark \ref{ProblemGhys} below. Most importantly, the fact that $\rho_1$ is left-semi-conjugate to $\rho_2$ by itself 
does not imply semi-conjugacy. In fact, left-semi-conjugacy is not even an equivalence relation, since the trivial action is left-semi-conjugate to every circle action. This problem can also not be remedied by replacing 
left-semi-conjugacy by the equivalence relation it generates, since the latter relation is just the trivial relation in which any two circle actions are related, nor by excluding constant semi-conjugacies,
 since these are necessary for Theorem \ref{IntroMain} to hold.

However, it is rather remarkable that for \emph{fixed point free circle actions} all these problems disappear completely. 
 In fact, as an immediate consequence of Theorem~\ref{ThmEquivalence} we have the following:

\begin{cor}\label{fpfree} If $\rho_1: \Gamma \to H$ and $\rho_2 : \Gamma \to H$ are fixed point free circle actions of the same group $\Gamma$, then the following are equivalent:
\begin{enumerate}
\item $\rho_1$ is semi-conjugate to $\rho_2$.
\item $\rho_1$ is left-semi-conjugate to $\rho_2$.
\item $\rho_1$ is right-semi-conjugate to $\rho_2$.
\item $\rho_1$ and $\rho_2$ have the same bounded Euler class.
\end{enumerate}
\end{cor}

This corollary is the reason why the wrong definitions in the literature are in most cases rather innocuous. Another issue concerning the definition of semi-conjugacy concerns the regularity of the non-decreasing degree one maps involved. As stated in Theorem \ref{ThmEquivalence}, if $\rho_1$ and $\rho_2$ are
 semi-conjugate circle actions, then one can find an upper semi-continuous left-semi-conjugacy from $\rho_1$ to $\rho_2$ (and vice versa). However, one can in general \emph{not} find a continuous left-semi-conjugacy
 from $\rho_1$ to $\rho_2$.  Nevertheless, semi-conjugacy may be defined via the use of continuous maps of Hopf-Brouwer degree $1$ rather than 
 (possibly non-continuous) non-decreasing degree one map as follows:
 \begin{thm}[\cite{Calegari_Book}]\label{Cal} Semi-conjugacy is the equivalence relation generated by continuous left-semi-conjugacies of Brouwer-Hopf degree $1$.
\end{thm}
Note that what we call a ``left-semi-conjugacy via a continuous map of Brouwer-Hopf degree 1'' here, is simply called a \emph{semi-conjugacy} in \cite{Calegari_Book}, conflicting with our terminology. On the other hand, the equivalence relation generated by continuous left-semi-conjugacies of Brouwer-Hopf degree $1$ which is equivalent to what we call ``semi-conjugacy'' is called \emph{monotone equivalence} in  \cite{Calegari_Book}. 
 


The rough outline of this note is as follows: In Section \ref{SecDef} we discuss the symmetric definition of semi-conjugacy stated as Definition (i) in Theorem \ref{ThmEquivalence}. 
In particular, we discuss the geometry of non-decreasing degree one maps and various pitfalls of the definition. Section \ref{SecBEC} is then devoted to the discussion of the bounded Euler class alluded to earlier. 
In particular, we discuss thoroughly three well-known characterizations of the bounded Euler class on $\omeo$ and establish carefully their mutual equivalence (Definition \ref{defBddEulerclass}, Proposition \ref{TNumber} 
and Corollary \ref{cor: e as cocycle on S1}). 

Section \ref{SecGhysThm} is the core of this note. Here we establish Theorem \ref{IntroMain} for our symmetric definition of semi-conjugacy (i.e. the equivalence (i)$\Leftrightarrow$(v) in Theorem \ref{ThmEquivalence}). 
It turns out that the argument for fixed point free actions and for actions with fixed points is quite different. Thus we first establish in Subsection \ref{SubsecVanishing} that a circle action has a fixed point if and only 
if it has vanishing bounded Euler class, and that this corresponds precisely to being semi-conjugate to the trivial circle action. This reduces the proof of Theorem \ref{IntroMain} to the case of fixed point free actions. 
For such actions we then establish that they are left-semi-conjugate if and only if they have the same bounded Euler class. This proves Theorem \ref{IntroMain} and at the same time yields 
the equivalences (i)$\Leftrightarrow$(ii)$\Leftrightarrow$(v) in Theorem \ref{ThmEquivalence}. 

Once Theorem \ref{IntroMain} is established, Theorem \ref{GhysMinimal} follows easily. This is explained in the final Subsection \ref{Section Min Actions}, 
of Section \ref{SecGhysThm}. 

In Section \ref{SecVariations} we collect various consequences of Ghys' Theorem. Firstly, we explain how Poincar\'e's classification of $\Z$-actions on the circle can be considered as a special case of Ghys' Theorem. Secondly, 
we deduce from Ghys' Theorem that every action of an amenable group on the circle is semi-conjugate to an action by rotations, a result commonly attributed to Hirsch and Thurston (see~\cite{HirschThurston} and~\cite[Theorem 2.79]{Calegari_Book}). 
Finally, we characterize circle actions 
with vanishing real bounded Euler class. The final Section \ref{SecEquivalence} is devoted to the proofs of Theorem \ref{ThmEquivalence}, Corollary \ref{fpfree} and Theorem \ref{Cal}. Finally, in the appendix, 
we discuss the pullback of the Euler class to the double covering group of $\omeo$. We show that this pullback can be represented by a multiple of the so-called Sullivan cocycle which has stronger vanishing properties and also 
generalizes nicely to higher dimensions.

Let us emphasize that we do not claim any originality for the proofs of Theorem \ref{GhysMinimal} and Theorem \ref{IntroMain} (whereas we believe Theorem \ref{thm: appendix} to be new). We hope that our presentation will help to make 
Ghys' beautiful theory of the bounded Euler class more accessible.\\

{\bf Acknowledgements.} The authors are indebted to Pierre de la Harpe
for a careful reading of a preliminary version of this note. They also thank  Danny Calegari, \'Etienne Ghys, Kathryn Mann, Daniel Monclair and Maxime Wolff for useful comments. The first author was supported by 
Swiss National Science Foundation project PP00P2-128309/1. The first and second authors are grateful to the American Institute of Mathematics for their support during the preparation of this work. 
The third author acknowledges support from the Taub Foundation within the Leaders in Science and Technology program.

\section{On the definition of semi-conjugacy}\label{SecDef}

\subsection{Non-decreasing degree one maps}\label{NonDecreasingDegreeOneMap}
Throughout this article we consider the circle $S^1 = \R/\Z$ as a quotient of the real line. A pre-image $\widetilde{x}$ of a point $x \in S^1$ under the canonical projection $\R \to S^1$ will be called a \emph{lift} of $x$ and we write $[\widetilde{x}] := x$.
\begin{defn} For $k\in \mathbb{N}$, an ordered $k$-tuple $(x_1,\dots,x_k)\in (S^1)^k$ is said to be
\begin{itemize}
\item  \emph{weakly positively oriented} if there exist lifts $\widetilde{x_i}\in \R$ of the $x_i$'s such that
$$\widetilde{x_1} \leq \widetilde{x_2}\leq \dots \leq \widetilde{x_k} \leq \widetilde{x_1}+1,$$
\item \emph{positively oriented} if furthermore 
$$ \widetilde{x_1} < \widetilde{x_2}<\dots < \widetilde{x_k} < \widetilde{x_1}+1.$$
\end{itemize}
Replacing $\leq$, $<$ and $\widetilde{x_1}, \widetilde{x_1}+1$ respectively by $\geq$, $>$ and $\widetilde{x_1}+1, \widetilde{x_1}$ we obtain the corresponding notion of (weakly) negatively oriented $k$-tuples. 
\end{defn}

Note that if $k\leq 2$ then a $k$-tuple is both weakly positively oriented and weakly negatively oriented. Furthermore,  the property of being (weakly) positively oriented is obviously invariant under cyclic permutations. 

\begin{defn}\label{Noddom}
A (not necessarily continuous) map $\varphi\colon S^1\to S^1$ is a \emph{non-decreasing degree one map} if the following condition holds for all $k \in \mathbb N$: If $(x_1,\ldots,x_k)\in (S^1)^k$ is weakly positively oriented,
 then  $(\varphi(x_1),\dots,\varphi(x_k))$ is weakly positively oriented.
\end{defn}
As we will see in Lemma \ref{quadruple:lem} below it is actually enough to check the condition for $k=4$. Observe that non-decreasing degree one maps are closed under composition and that every constant map is a non-decreasing 
degree one map. 

\begin{defn} Let $\varphi\colon S^1\to S^1$ be any map. A set-theoretical lift $\widetilde{\varphi}\colon \R\to\R$ of $\varphi$ is called a \emph{good lift} of $\varphi$ if $\widetilde{\varphi}(x+1)=\widetilde{\varphi}(x)+1$ 
for every $x\in\R$ and $\widetilde{\varphi}$ is non-decreasing, i.e. $\widetilde{\varphi}(x)\leq \widetilde{\varphi}(y)$ whenever $x\leq y$. 
\end{defn}

By the following lemma, being a non-decreasing degree one map is equivalent to admitting a good lift, so Definition \ref{Noddom} is equivalent to the more classical definition which we used in the introduction. We warn the reader that a non-decreasing degree one map may have infinitely many essentially different good lifts, i.e. good lifts which do not just differ by 
composition with an integral translation. For example, for every $\alpha\in\R$ the maps $x\mapsto \lfloor x+\alpha\rfloor$ and $x\mapsto \lceil x+\alpha\rceil$ are good lifts of the constant map $\varphi\colon S^1\to S^1$ mapping every point to $[0]$.

\begin{lemma}\label{quadruple:lem} Let $\varphi\colon S^1\to S^1$ be any map. Then the following conditions are equivalent:
\begin{enumerate}[(i)]
\item The map $\varphi$ is a non-decreasing degree one map.
\item If $(x_1,\ldots,x_4)\in (S^1)^4$ is weakly positively oriented, then  $(\varphi(x_1),\dots,\varphi(x_4))$ is weakly positively oriented;
\item There exists a good lift of $\varphi$.
\end{enumerate}
\end{lemma}

\begin{proof} The implication (i) $\Rightarrow$ (ii) holds by definition. 

(ii) $\Rightarrow$ (iii): If $\varphi$ is constant, there is nothing to prove. Suppose there exist $x_0\neq x_1\in S^1$ such that $y_0:=\varphi(x_0)\neq \varphi(x_1)=:y_1$. 
Choose lifts $\widetilde{x_0},\widetilde{y_0},\widetilde{x_1},\widetilde{y_1}\in \R$ of $x_0,y_0,x_1,y_1$ respectively such that $\widetilde{x_1}\in (\widetilde{x_0},\widetilde{x_0}+1)$ 
and $\widetilde{y_1}\in (\widetilde{y_0},\widetilde{y_0}+1)$. Now define $\widetilde{\varphi}$ on  $[\widetilde{x_0},\widetilde{x_0}+1)$ as follows: for $\widetilde{x_0}\leq \widetilde{x} \leq \widetilde{x_1}$, 
let  $\widetilde{\varphi}(x)$ be the unique lift of $\varphi([\widetilde{x}])$ lying in $[\widetilde{y_0},\widetilde{y_0}+1)$; for $\widetilde{x_1}\leq \widetilde{x} <\widetilde{x_0}+1$, let $\widetilde{\varphi}(x)$ 
be the unique lift of $\varphi([\widetilde{x}])$ lying in $(\widetilde{y_0},\widetilde{y_0}+1]$. Now extend $\widetilde{\phi}$ to $\R$ in the unique possible way such that it commutes with integral translations.

In order to see that $\widetilde{\phi}$ is non-decreasing it suffices to show that it is non-decreasing on 
 $[\widetilde{x_0},\widetilde{x_0}+1)$. Thus let $\widetilde{x_0}\leq \widetilde{x}<\widetilde{y}< \widetilde{x_0}+1$. 
 
We first prove that if $\widetilde{\phi}(\widetilde{y})=\widetilde{y_0}$, then $\widetilde{\phi}(\widetilde{x})=\widetilde{y_0}$. Indeed 
$\widetilde{\phi}(\widetilde{y})$ can be equal to $\widetilde{y}_0$ only if $\widetilde{y}<x_1$. Thus the quadruple $(x_0, [\widetilde{x}], [\widetilde{y}], x_1)$ is weakly positively oriented, and so is $(\varphi(x_0),\varphi( [\widetilde{x}]),\varphi([\widetilde{y}]),\varphi(x_1)) = (y_0, [\widetilde{\phi}(\widetilde{x})], y_0,y_1)$ by (ii). By definition, this means that 
there exist integers $n_x,n_y,m\in \Z$ such that
 \[ \widetilde{y_0} \leq \widetilde{\varphi}(\widetilde{x})+n_x \leq \widetilde{y_0}+n_y \leq \widetilde{y_1}+m \leq \widetilde{y_0}+1\ .\]
 Since $\widetilde{y_1}\in (\widetilde{y_0},\widetilde{y_0}+1)$ we have $m=0$, and this implies in turn that $n_y=0$, so 
that  $\widetilde{\phi}(\widetilde{x})+n_x=\widetilde{y_0}$. But since $\widetilde{x}<\widetilde{x_1}$ its image by $\widetilde{\varphi}$ lies in $[\widetilde{y_0},\widetilde{y_0}+1)$, so 
  $\widetilde{\phi}(\widetilde{x})=\widetilde{y_0}$, as desired.

  A completely analogous symmetric argument also shows that if $\widetilde{\phi}(\widetilde{x})=\widetilde{y_0}+1$, then  $\widetilde{\phi}(\widetilde{y})=\widetilde{y_0}+1$. Thus we can now restrict to the case where $\widetilde{\phi}(\widetilde{x}),\widetilde{\phi}(\widetilde{y})\in (\widetilde{y_0},\widetilde{y_0}+1)$.
 
 From the assumption that $\widetilde{x_0}\leq \widetilde{x}<\widetilde{y}< \widetilde{x_0}+1$, we obtain that  
the quadruple $(x_0, [\widetilde{x}], [\widetilde{y}], x_0)$ is weakly positively oriented, and thus also the quadruple $(\varphi(x_0),\varphi([\widetilde{x}]),\varphi([\widetilde{y}]),\varphi(x_0)) = (y_0,[\widetilde{\phi}(\widetilde{x})],[\widetilde{\phi}(\widetilde{y})],y_0)$ is weakly positively oriented by (ii). By definition this means that 
there exist  integers $n_x,n_y,m\in \Z$ such that
\[\widetilde{y_0} \leq \widetilde{\varphi}(\widetilde{x})+n_x \leq \widetilde{\varphi}(\widetilde{y})+n_y \leq \widetilde{y_0}+m \leq \widetilde{y_0}+1\ .\]
Since $\widetilde{\varphi}(\widetilde{x})$ and $\widetilde{\varphi}(\widetilde{y})$ are now both belonging to the open interval $(\widetilde{y_0},\widetilde{y_0}+1)$ it follows that $n_x=n_y=0$ (and $m=1$).
We thus obtain $\widetilde{\varphi}(\widetilde{x})\leq  \widetilde{\varphi}(\widetilde{y})$, which finishes the proof of this implication.

(iii) $\Rightarrow$ (i): Let $x_0,\dots,x_k$ be weakly positively oriented. By definition this means that there exist  lifts $\widetilde{x_i}\in \R$ of the $x_i$'s such that
$$\widetilde{x_1} \leq \widetilde{x_2}\leq \dots \leq \widetilde{x_k} \leq \widetilde{x_1}+1.$$
Applying the non-decreasing map $\widetilde{\varphi}$ to the above inequalities gives
$$\widetilde{\varphi}(\widetilde{x_1}) \leq \widetilde{\varphi}(\widetilde{x_2})\dots \leq\widetilde{\varphi}( \widetilde{x_k}) \leq \widetilde{\varphi}(\widetilde{x_1}+1)=\widetilde{\varphi}(\widetilde{x_1})+1,$$
where the last equality uses the fact that $\widetilde{\varphi}$ commutes with integral translations. Since the $\widetilde{\varphi}(x_i)$'s are lifts of $\varphi(x_i)$, this by definition implies that the $k$-tuple $(\varphi(x_1),\dots,\varphi(x_k))$ is weakly positively oriented. 
\end{proof}

It is clear from the proof that we cannot replace the statement in (ii) with the corresponding statement for triples. To give an explicit counterexample, consider the function $\varphi:S^1 \to S^1$ given by
\[
\varphi([t]) = \left\{\begin{matrix}[0], &  t  \in [0,1/4)\cup [1/2,3/4),\\
                                                   [1/2], & t  \in [1/4,1/2)\cup [3/4,1). \end{matrix} \right.
\]
This function $\varphi$ takes any triple into a weakly positively oriented one, but the quadruple $([0],[1/4],[1/2],[3/4])$ is taken by $\varphi$ to $([0],[1/2],[0],[1/2])$, which is not weakly positively oriented. 

\subsection{Semi-conjugacy} Let us recall the key definition of this note from the introduction.
\begin{defn}\label{semi:def}\label{DefSc} \label{MainDef}
Let $\rho_j\colon \G\to\omeo$ be circle actions, $j=1,2$. We say that $\rho_1$ is \emph{left-semi-conjugate} to $\rho_2$ (and $\rho_2$ is right-semi-conjugate to $\rho_1$) if
there exists a non-decreasing degree one map $\varphi$ such that \[\rho_1(\gamma)\varphi=\varphi\rho_2(\gamma)\] for every $\gamma\in \G$. In this case, $\varphi$ is called a \emph{left-semi-conjugacy} from $\rho_1$ to $\rho_2$ and we say that $\rho_1$ is left-semi-conjugate to $\rho_2$ \emph{via $\varphi$}. 

The circle action $\rho_1$ is called \emph{semi-conjugate} to $\rho_2$ if it is both left- and right-semi-conjugate to $\rho_2$.
\end{defn}

We recall some standard terminology for group actions: A circle action $\rho:\Gamma \rightarrow \omeo$ is said to have a \emph{global fixed point} if there exists $x\in S^1$ such that $\rho(\gamma)(x)=x$ for every $\gamma\in \Gamma$. An action is \emph{fixed point free} if it does not admit a global fixed point. 

\begin{prop}\label{semiconj:rem}
\begin{enumerate}[(i)]
\item Semi-conjugacy is an equivalence relation.
\item Every circle action is right-semi-conjugate to the trivial action.
\item A circle action is left-semi-conjugate to the trivial action if and only if it has a global fixed point.
\end{enumerate}
\end{prop}
\begin{proof} (i) Reflexivity and symmetry are obvious,  while transitivity readily follows from the fact that non-decreasing degree one maps are closed under composition. (ii) Choose $\varphi$ to be an arbitrary constant map. (iii) If $\rho$ is left-semi-conjugate to the trivial action, then 
there exists $\varphi$ such that for all $\gamma \in\G$ and $x\in S^1$
\[\rho({\gamma})(\varphi(x))=\varphi(x)\,\]
whence the image of $\varphi$ consists of fixed points of $\rho(\G)$. 
On the other hand, if $x_0$ is fixed by $\rho(\G)$, then $\rho$ is left-semi-conjugate to the trivial action by the constant map $\varphi(x) \equiv x_0$.\end{proof}

\begin{rem}\label{ProblemGhys}
The definition of semi-conjugacy given in \cite{Ghys0} coincides with our definition of left-semi-conjugacy. As it obviously follows from Proposition \ref{semiconj:rem} (ii)-(iii) that left-semi-conjugacy is not even an equivalence relation, it cannot be the correct notion. However, for \emph{fixed point free} circle actions it does indeed coincide with our notion of semi-conjugacy, see Corollary \ref{mistake:small}.


Elsewhere in the literature semi-conjugacy is defined as the existence of a \emph{continuous} left semi-conjugacy $\varphi:S^1\rightarrow S^1$. This is still not symmetric: as we saw in the proof of Proposition~\ref{semiconj:rem} (ii), 
every circle action is right semi-conjugate to the trivial action via a continuous map $\varphi$, while by point (iii) of the same proposition fixed point free actions cannot be left-semi-conjugate to the trivial action. 

Since constant left-semi-conjugacies are responsible for both problems, one may be tempted to exclude them from the game. Such a more restrictive definition does indeed appear in the literature, but Theorem \ref{IntroMain} can never hold for such a definition. Namely, it is easy to check that if $\rho_1$ admits a unique global fixed point $x_0$ and $\rho_2$ is the trivial representation, then the constant map with image $\{x_0\}$ 
is the unique left-semi-conjugacy from $\rho_1$ to $\rho_2$. On the other hand $\rho_1$ and $\rho_2$ have the same bounded Euler class (see Corollary~\ref{zeroprop} below), so they need to be semi-conjugate in order for Theorem \ref{IntroMain} to hold.

In some sense, semi-conjugacy in the sense of Definition \ref{MainDef} is the most obvious way to turn left-semi-conjugacy into an equivalence relation. However, contrary to what is sometimes claimed, it is \emph{not} the equivalence relation generated by left-semi-conjugacy. Namely, by Proposition \ref{semiconj:rem} the equivalence relation generated by left-semi-conjugacy is the trivial relation in which any two circle actions are related.

By definition, conjugate circle actions are semi-conjugate. We will see in Proposition \ref{dense:orbits} below that for minimal circle actions the converse holds. However, in general the notion of semi-conjugacy is much weaker than the notion of conjugacy. For example Proposition~\ref{semiconj:rem} shows that every circle action admitting a fixed point is semi-conjugate to the trivial circle action 
(but of course not conjugate to the trivial circle action unless it is trivial itself).
\end{rem}

\section{Three characterizations of the bounded Euler class}\label{SecBEC}

The goal of this section is to introduce the bounded Euler class and provide three different characterizations: as a bounded obstruction class (Subsection \ref{SubsecEuler1}), via the translation number (Subsection \ref{SubsecEuler2}) 
and as a bounded geometric class on the circle (Subsection \ref{SubsecEuler3}). Yet another description of the bounded Euler class, which generalizes readily to higher dimensions, will be discussed in the appendix. In order to keep 
this note self-contained we collect in the next subsection various basic facts concerning (bounded) group cohomology. The expert can skip that subsection without loss of continuity.

\subsection{Preliminaries on (bounded) group cohomology} 
Given a group $H$ acting on a space $X$ we set $\mathcal C^n(H\curvearrowright X; \Z) := {\rm Map}(X^{n+1}; \Z)^H$, where the superscript ${}^H$ denotes $H$-invariants under the diagonal $H$-action, and refer to elements of $\mathcal C^n(H\curvearrowright X; \Z)$ as \emph{homogeneous $H$-cochains} of degree $n$ (or simply a \emph{homogeneous cohain} if $H$ is clear from the context). We then obtain a cocomplex $(\mathcal C^n(H\curvearrowright X; \Z), \delta)$ by defining the \emph{homogeneous differential} $\delta$ as
\[\delta f(x_0,\dots,x_n)=\sum_{i=0}^n (-1)^i f(x_0,\dots,\widehat{x_i},\dots,x_n),\]
whose cohomology we denote by $H^\bullet(H\curvearrowright X; \Z)$. Elements in the kernel, respectively image of $\delta$ are called \emph{homogeneous $H$-cocycles}, respectively \emph{homogeneous $H$-coboundaries}.
If $X = H$ with the left-$H$-action, then the cohomology $H^\bullet(H\curvearrowright X; \Z)$ is precisely the classical group cohomology $H^\bullet(H; \Z)$ with $\Z$-coefficients. Given a homogeneous cocycle $c \in \mathcal C^n(H\curvearrowright X;\Z)$ and a basepoint $x_0 \in X$ we obtain a homogeneous cocycle $c_{x_0} \in \mathcal C^n(H\curvearrowright H;\Z)$ by
\[
c_{x_0}(h_0, \dots, h_n) = c(h_0\cdot x_0, \dots, h_n\cdot x_0).
\]
The class of $c_{x_0}$ is independent of the choice of basepoint $x_0$. We thus obtain a map $\iota_X: H^\bullet(H\curvearrowright X; \Z) \to H^\bullet(H; \Z)$ and we say that a class $\alpha \in H^\bullet(H; \Z)$ is \emph{represented over $X$} if it is in the image of this map.

There is a more efficient representation for classes in $H^\bullet(H; \Z)$ based on the fact that we can identify $\mathcal C^n(H\curvearrowright H; \Z)$ with $C^n(H;\Z) := {\rm Map}(H^n;\Z)$ via the isomorphism
\[
\iota: C^n(H; \Z) \to \mathcal C^n(H\curvearrowright  H; \Z)
\]
given by
\begin{eqnarray*}
\iota(f)(h_0,\dots,h_n) &:=& f(h_0^{-1}h_1,h_1^{-1}h_2,\dots, h_{n-1}^{-1}h_n)\\
\iota^{-1}(g)(h_1, \dots, h_n) &:=& g(e,h_1,h_1h_2,\dots,h_1h_2\cdot \ldots \cdot h_n ).
\end{eqnarray*}
Thus $H^\bullet(H; \Z) = H^\bullet(C^\bullet(H; \Z), d)$, where the differential $d = \iota^{-1}\circ \delta \circ \iota$ is given by
\begin{eqnarray*}
df(h_1, \dots, h_{n+1}) &=& f(h_2, \dots, h_{n+1}) + \sum_{i=1}^n (-1)^i f(h_1, \dots, h_ih_{i+1}, \dots, h_{n+1})\\
&&+(-1)^{n+1}f(h_1, \dots, h_n).
\end{eqnarray*}  
Cochains in this model are called \emph{inhomogeneous cochains}, and are particularly useful to compute low degree cohomologies. We will be specifically interested in cohomology of degree $2$; we thus recall briefly the relation between $H^2(H; \Z)$ and central extensions. Given a central extension of groups of the form
\[\xi =\left(\begin{xy}
\xymatrix{
0\ar[r]&\Z \ar[r]^i&\widetilde{H} \ar[r]^p&H \ar[r]&\{e\}
}
\end{xy}\right)\]
and a set theoretic section $\sigma: H \to \widetilde{H}$ of $p$ we define a function $c_\sigma: H^2 \to \widetilde{H}$ by 
\[
c_\sigma(h_1,h_2) = \sigma(h_2)\sigma(h_1h_2)^{-1}\sigma(h_1).
\]
Since $p(c_\sigma(h_1,h_2)) = e$ we can consider $c_\sigma$ as a function into $i(\Z)$. We will often tacitly identify $\Z$ with its image in $\widetilde{H}$ and thus consider $c_\sigma$ as a function $c_\sigma: H^2 \to \Z$. It is straightforward to check that $c_\sigma$ satisfies the cocycle identity
\[
dc_\sigma(h_1,h_2,h_3) = c_\sigma(h_2,h_3) -c_\sigma(h_1h_2,h_3)+c_\sigma(h_1, h_2h_3) - c_\sigma(h_1,h_2)=0,
\] 
whence we refer to it as the \emph{obstruction cocycle} associated with the extension $\xi$ and the section $\sigma$. It turns out that the class $e(\xi) := [c_\sigma] \in H^2(H;\Z)$ is independent of the choice of section. This independence can easily be proved directly, but it is also a consequence of the following universal property of the class $[c_\sigma]$: 
\begin{lemma}[Lifting obstruction]\label{Lift}
If $\rho: \Gamma \to H$ is a homomorphism, then there exists a lift
\[\begin{xy}
\xymatrix{
0\ar[r]&\Z\ar[r]^i&\widetilde{H} \ar[r]^p&H \ar[r]&\{e\}\\
&&&\Gamma \ar[u]_\rho \ar@{.>}[ul]^{\widetilde{\rho}}&
}
\end{xy}\]
if and only if $\rho^*[c_\sigma] = 0 \in H^2(\Gamma;\Z)$.
\end{lemma}
Conversely, a class  $e\in H^2(\G;\Z)$ determines a central extension, which is unique up to a suitable notion of isomorphism between
extensions. We refer the reader to \cite[Chapter IV]{brown} for the details.

In the sequel we will need the following explicit version of (one direction of) the lemma:
\begin{prop}[Lifting formula]\label{LiftExplicit}
Let $\rho: \Gamma \to H$ be a homomorphism. Assume that $\rho^*c_\sigma = du$ for some $u:\Gamma \to \mathbb Z$. Then a homomorphic lift $\widetilde{\rho}: \Gamma \to \widetilde{H}$ is given by the formula
\[
\widetilde{\rho}(\gamma) = \sigma(\rho(\gamma)) \cdot i(-u(\gamma)).
\]
\end{prop}
\begin{proof} Since this formula is at the heart of our argument we carry out the straightforward computation. By definition of $c_\sigma$, we have
\begin{equation}\label{equ lemma 3.2}
\sigma(\rho(\gamma_1\gamma_2))=\sigma(\rho(\gamma_1)\rho(\gamma_2))=\sigma(\rho(\gamma_1))c_\sigma(\rho(\gamma_1,\rho(\gamma_2))^{-1}\sigma(\rho(\gamma_2)).
\end{equation}
Since by assumption $\rho^*c_\sigma = du$, we have
$$c_\sigma(\rho(\gamma_1),\rho(\gamma_2))=\rho^*(c_\sigma)(\gamma_1,\gamma_2)=du(\gamma_1,\gamma_2)=u(\gamma_2)-u(\gamma_1\gamma_2)+u(\gamma_1).$$
Since $i(\Z)$ is central in $\widetilde{H}$ we can rewrite Equation (\ref{equ lemma 3.2}) as
$$\sigma(\rho(\gamma_1\gamma_2))=\sigma(\rho(\gamma_1))i(-u(\gamma_1)\sigma(\rho(\gamma_2)(i(-u(\gamma_2))i(u(\gamma_1\gamma_2).$$
Multiplying both sides by $ i(-u(\gamma_1\gamma_2))$ now yields $\widetilde{\rho}(\gamma_1\gamma_2) =\widetilde{ \rho}(\gamma_1)\widetilde{\rho}(\gamma_2)$ and finishes the proof.
\end{proof}

The subcomplex $\mathcal C^n_b(H\curvearrowright X; \Z) \subset \mathcal C^n(H \curvearrowright X; \Z)$ of bounded functions is invariant under $\delta$, and its cohomology is called the \emph{(integral) bounded cohomology} of the $H$-action on $X$ and denoted $H_b^\bullet(H\curvearrowright X; \Z)$. In particular, $H_b^\bullet(H; \Z) := H_b^\bullet(H\curvearrowright H; \Z)$ is the bounded group cohomology of $H$ in the sense of \cite{Gromov}. 
Note that the isomorphism $\iota: C^n(H; \Z) \to \mathcal C^n(H\curvearrowright H; \Z)$ identifies $C_b^n(H\curvearrowright H; \Z)$ with the subspace $C^n_b(H; \Z) < C^n(H; \Z)$ of bounded functions, hence $H_b^\bullet(H; \Z)$ 
can also be computed from bounded inhomogeneous cochains.

The inclusion of complexes $(C^n_b(H; \Z), \delta) \hookrightarrow (C^n(H; \Z), \delta)$ induces on the level of cohomology a  \emph{comparison map} $H^\bullet_b(H;\Z) \to H^\bullet(H;\Z)$, whose kernel is classically 
denoted by $EH^\bullet_b(H;\Z)$. Note that an inhomogeneous bounded cocycle representing a class in  $EH^2_b(H;\Z)$ is of the form $dT$ for some $T: H \to \Z$ with the property that $|T(h_1h_2)-T(h_1)-T(h_2)| = |dT(h_1,h_2)|$ is 
uniformly bounded. Such a function $T$ is called an integral \emph{quasimorphism} and the number $D(T) := \|dT\|_\infty$ is called its \emph{defect}. Given two quasimorphisms $T_1, T_2$ we have $[dT_1] = [dT_2] \in EH^2_b(H;\Z)$ 
if and only if $T_1-T_2 \in {\rm Hom}(H;\Z)\oplus {\rm Map}_b(H; \Z)$. In particular, changing $T$ by a bounded amount does not change the bounded cohomology class of $[dT]$.

Bounded group cohomology can also be defined with real coefficients. In this case, bounded inhomogeneous cocycles in $EH^2_b(H;\R)$ are of the form $dT$ where $T$ is a real-valued quasimorphism. Every real-valued quasimorphism 
(and in particular every integral one) is at bounded distance from a unique homogeneous real-valued quasimorphism called its \emph{homogeneization}. Here a real-valued function $f$ is called \emph{homogeneous} 
provided $f(h^n) = n\cdot f(h)$ for all $n \in \mathbb N$. Homogeneous quasimorphisms have the additional properties of being conjugacy-invariant and linear on abelian subgroups. They also satisfy $f(h^n) = n\cdot f(h)$ 
for all $n\in \Z$, positive or not. Note that two quasimorphisms are at bounded distance if and only if their homogeneizations coincide. The following lemma illustrates how bounded cohomology with real coefficients 
can be used to obtain results concerning integral bounded cohomology; we will apply this in our second characterization of the bounded Euler class below.
\begin{lemma}\label{InjectivityExtension} If $p: \widetilde{H} \to H$ is a surjective homomorphism with amenable (e.g. abelian) kernel, then $p^*: H^2_b(H; \Z) \to H^2_b(\widetilde{H}; \Z)$ is injective.
\end{lemma}
\begin{proof} The short exact sequence $0 \to \Z \to \R \to \R/\Z \to 0$ of coefficients induces a natural long exact sequence in bounded cohomology, called the Gersten sequence (see \cite[Prop. 8.2.12]{Monod}), and the corresponding ladder associated with the homomorphism $p$ starts from
\[\begin{xy}\xymatrix{
0 \ar[r]\ar[d]& {\rm Hom}(H; \R/\Z)\ar[d]_{p^*}\ar[r] & H^2_b(H; \Z) \ar[d]_{p^*}\ar[r]&H^2_b(H; \R)\ar[d]_{p^*}\\
0 \ar[r]& {\rm Hom}(\widetilde{H}; \R/\Z) \ar[r]& H^2_b(\widetilde{H}; \Z)\ar[r] &H^2_b(\widetilde{H}; \R)\\
}\end{xy}
\]
Now surjectivity of $p$ implies that the pullback map $p^*: {\rm Hom}(H; \R/\Z) \to  {\rm Hom}(\widetilde{H}; \R/\Z)$ is injective, and the  map $p^*:H^2_b(H; \R) \to H^2_b(\widetilde{H}; \R)$ is an isomorphism 
by \cite{Gromov, Ivanov}, whence the lemma follows from the $4$-lemma.
\end{proof}

\subsection{The bounded Euler class as a bounded lifting obstruction}\label{SubsecEuler1}
From now on we reserve the letter $H$ to denote the group  $H := {\rm Homeo}^+(S^1)$ of orientation-preserving homeomorphisms of the circle $S^1 = \R/\Z$ and abbreviate by
\[
\widetilde{H} := \{\widetilde{h} \in {\rm Homeo}^+(\R)\mid \forall x \in \R:\; \widetilde{h} (x+1) = \widetilde{h} (x)+1\}
\]
its universal covering group (with respect to the compact-open topology). We then have a central extension
\[\xi =\left(\begin{xy}
\xymatrix{
0\ar[r]&\Z \ar[r]^i&\widetilde{H} \ar[r]^p&H \ar[r]&\{e\}
}
\end{xy}\right),\]
where $i(n)(x) := x+n$ and $p(\widetilde{h})([x]) = [\widetilde{h}(x)]$.

A section $\sigma: H \to \widetilde{H}$ is provided by specifying $\sigma(h)(0)$ for each $h \in H$; the section is called \emph{bounded} provided $E_\sigma := \{\sigma(h)(0)\mid h \in H\}$ is bounded. In this case the obstruction cocycle $c_\sigma: H^2 \to \Z$ is bounded and thus defines also a class in the \emph{bounded} second cohomology $H^2_b(H;\Z)$. Again it is easy to see that this class is independent of the choice of bounded section.
 We then obtain two classes ${\rm eu} :=[c_\sigma] \in H^2(H;\Z)$ and ${\rm eu}_b :=[c_\sigma] \in H^2_b(H;\Z)$.
\begin{defn} \label{defBddEulerclass} The classes {\rm eu} and ${\rm eu}_b$ are called the \emph{Euler class}, respectively \emph{bounded Euler class}.
\end{defn}
One special section $\sigma$ is obtained by taking $E_\sigma \subseteq [0,1)$. Let us give an explicit formula for the cocycle $c_\sigma$ in this case. For all $h_1, h_2\in H$ we have $\sigma(h_2)\sigma(h_1h_2)^{-1}\sigma(h_1) = i(c_\sigma(h_1,h_2))$. Since $i(\Z)<\widetilde{H}$ is central this can be written as $\sigma(h_1)\sigma(h_2) = \sigma(h_1h_2)i(c_\sigma(h_1,h_2))$. Evaluating at $0$ we obtain
\[
\sigma(h_1)\sigma(h_2)(0) = \sigma(h_1h_2)(0) + c_\sigma(h_1,h_2).
\]
Observe that $\sigma(h_1h_2)(0)$ and $\sigma(h_2)(0)$ are contained in $[0,1)$. The latter implies that $\sigma(h_1)\sigma(h_2)(0) \in [0,2)$. Thus
\begin{equation}\label{FirstFromulacsigma}
c_\sigma(h_1,h_2) =\left\{\begin{array}{ll} 1 & \textrm{if}\ \sigma(h_1)\sigma(h_2)(0) \in [1,2),  \\0 & \textrm{if}\ \sigma(h_1)\sigma(h_2)(0) \in [0,1).\end{array}\right.
\end{equation}
Another equivalent description can be given as follows: Observe that $\sigma(h_1)(1) = \sigma(h_1)(0) +1 \in [1,2)$ and that $\sigma(h_2)(0) < 1$ implies $\sigma(h_1)\sigma(h_2)(0)<\sigma(h_1)(1)$, 
and similarly $0 \leq \sigma(h_2)(0)$ implies $\sigma(h_1)(0)\leq \sigma(h_1)\sigma(h_2)(0)$. We may thus rewrite \eqref{FirstFromulacsigma} as
\begin{equation}\label{csigmaExplicit}
c_\sigma(h_1,h_2) =\left\{\begin{array}{ll} 1 & \textrm{if}\ 1\leq \sigma(h_1)\sigma(h_2)(0)<\sigma(h_1)(1)<2,  \\0 & \textrm{if}\ 0\leq \sigma(h_1)(0)\leq \sigma(h_1)\sigma(h_2)(0)<1. \end{array}\right.
\end{equation}
Both formulas will be used below.

\subsection{The bounded Euler class and the translation number}\label{SubsecEuler2}
The \emph{Poincar\'e translation number} $T: \widetilde{H} \to \R$ is the homogeneous quasimorphism on $\widetilde{H}$ given by
\[
{T}(\widetilde{h}) = \lim_{n \to \infty}\frac{\widetilde{h}^nx-x}{n}\quad(x \in \R),
\]
which by a classical theorem of Poincar\'e is independent of the choice of basepoint $x \in \R$ (see \cite{Poi1, Poi2}). Let $T_\Z: \widetilde{H} \to \Z$ be any function at bounded distance from $T$. Then the cocyle $dT_\Z$ is bounded and thus defines a 
class $[dT_\Z] \in H^2_b(\widetilde{H}; \Z)$, which is independent of the concrete choice of function $T_\Z$. We can now state the second characterization of the bounded Euler class. We recall that $p: \widetilde{H} \to H$ denotes the canonical projection.
\begin{prop}\label{TNumber} The bounded Euler class ${\rm eu}_b$ is the unique class in $H^2_b(H; \Z)$ such that $p^*{\rm eu}_b = -[dT_\Z] \in H^2_b(\widetilde{H};\Z)$.
\end{prop}
\begin{proof} Let $\widetilde{h_1},\widetilde{h_2} \in \widetilde{H}$. We abbreviate $h_1 := p(\widetilde{h_1})$, $h_2 := p(\widetilde{h_2} )$. Given a real number $r \in \R$ we denote by $r = \lfloor r \rfloor + \{r\}$ the unique decomposition of $r$ with $\lfloor r\rfloor \in \Z$ and $\{r\} \in [0,1)$. Since $\widetilde{h_1}$ and $\sigma(h_1)$ have the same projection they differ by an integral translation which we obtain by evaluating the difference on $0$. We thus compute
\[
 \widetilde{h_1} (0) - \sigma(h_1)(0) = \lfloor \widetilde{h_1}(0) \rfloor+\{\widetilde{h_1}(0)\} -  \sigma(h_1)(0) = \lfloor \widetilde{h_1}(0) \rfloor,
\]
 where the last equality follows from that both $\{\widetilde{h_1}(0)\}$ and $ \sigma(h_1)(0)$ belong to $[0,1)$. Thus, for every $x\in \mathbb{R}$, we have
$\sigma(h_1)(x) = \widetilde{h_1}(x) -\lfloor \widetilde{h_1}(0) \rfloor$ and similarly $\sigma(h_2)(x) = \widetilde{h_2} (x) -\lfloor \widetilde{h_2} (0) \rfloor$. We deduce that
\begin{eqnarray*}
\sigma(h_1)\sigma(h_2)(0) &=& \sigma(h_1)(\widetilde{h_2} (0) -\lfloor \widetilde{h_2} (0) \rfloor) = \sigma(h_1)(\widetilde{h_2} (0)) - \lfloor \widetilde{h_2} (0) \rfloor\\
&=& \widetilde{h_1}\widetilde{h_2} (0) -\lfloor \widetilde{h_1}(0) \rfloor -\lfloor \widetilde{h_2} (0) \rfloor\\
&=& \lfloor \widetilde{h_1}\widetilde{h_2} (0)\rfloor -\lfloor \widetilde{h_1}(0) \rfloor -\lfloor\widetilde{h_2} (0) \rfloor + \{\widetilde{h_1}\widetilde{h_2} (0)\}.
\end{eqnarray*}
Since the last term is contained in $[0,1)$, this expression is in $[1,2)$ respectively $[0,1)$ if the sum of the first three terms is equal to $1$ respectively $0$. Representing $\mathrm{eu}_b$ by the cocycle $c_\sigma$ given in \eqref{FirstFromulacsigma}, we thus obtain
\[
p^*c_\sigma(\widetilde{h_1},\widetilde{h_2} ) =  c_\sigma(h_1,h_2)=\lfloor \widetilde{h_1}\widetilde{h_2} (0)\rfloor -\lfloor \widetilde{h_1}(0) \rfloor -\lfloor \widetilde{h_2} (0) \rfloor.
\]
Now the function $T_\Z: \widetilde{H} \to \Z$ given by $\widetilde{h_1} \mapsto \lfloor \widetilde{h_1}(0) \rfloor$ is at bounded distance from the translation number $T$ and the last identity can be written as $p^*c_\sigma = -dT_\Z$. We thus deduce that $p^*{\rm eu}_b = -[dT_\Z]$ and uniqueness follows from Lemma \ref{InjectivityExtension}.
\end{proof}

\subsection{The bounded Euler class realized over the circle}\label{SubsecEuler3}\label{explicit cocycles}
In this subsection we are going to show that the Euler class and the bounded Euler class are representable over the circle, i.e. that they are in the respective images of the maps $H^2(H\curvearrowright S^1;\Z) \to H^2(H; \Z)$ 
and $H^2_b(H\curvearrowright S^1;\Z) \to H^2_b(H; \Z)$. Recall that throughout we think of $S^1$ as the quotient space $\R/\Z$. In order to describe cocycles in $\mathcal C^n(H\curvearrowright S^1; \Z)$ we need to understand $H$-orbits in $(S^1)^{n+1}$. For $n \leq 2$ the classification of orbits is as follows: 
\subsection*{ Orbits of $H$ acting on $(S^1)^{n+1}$}
\begin{itemize}
\item[(n=0)] The action of $H$ on $S^1$ has exactly one orbit. 
\item[(n=1)] The action of $H$ on $(S^1)^2$ has two orbits: one degenerate orbit $\mathcal O_{deg} := \{(x,x)\mid x\in S^1\}$ and one non-degenerate orbit $\mathcal O_{ndeg} := \{(x,y)\mid x\neq y \in S^1\}$. 
\item[(n=2)]The action of $H$ on three factors $(S^1)^3$ has six orbits. Choose distinct points $x,y,z\in S^1$ and suppose that $(x,y,z)$ is a positively oriented triple. Then there are $4$ degenerate orbits 
\[\mathcal O_0 := H\cdot (x,x,x),\quad \mathcal O_1 := H\cdot (y,x,x),\quad \mathcal O_2 := H\cdot (x,y,x), \quad \mathcal O_3 :=H\cdot (x,x,y),\]
and $2$ non-degenerate orbits
\[ \mathcal O_+ := H\cdot (x,y,z), \quad \mathcal O_- := H\cdot (y,x,z).\]
\end{itemize}
For general $n$ there are still only finitely many $H$-orbits in $(S^1)^n$. This implies $\mathcal C^n_b(H\curvearrowright S^1; \Z) = \mathcal C^n(H\curvearrowright S^1; \Z) $ and thus the comparison map $H^n_b(H\curvearrowright S^1;\Z) \cong H^n(H\curvearrowright S^1; \Z)$ is an isomorphism. In particular, if an element of $H^n(H; \Z)$ is representable over $S^1$, then it is bounded.

In degree $2$ we can actually parametrize all possible homogeneous $H$-cocycles and homogeneous $H$-coboundaries using the above enumeration of orbits. Note that every homogeneous $2$-cochain $f$ is determined by the $6$ integers $\{f_0, f_1, f_2, f_3, f_+, f_-\}$, where $f_j$ is the value of $f$ on the orbit $\mathcal O_j$ for $j \in \{0,1,2,3,+,-\}$. For homogeneous coboundaries a straightforward computations shows that these numbers are given as follows.
\begin{lemma} \label{lemma b coboundary} Let $b:(S^1)^2\rightarrow \mathbb{R}$ be an arbitrary homogeneous $1$-cochain taking the values $\alpha$ and $\beta$ on the orbits $\mathcal O_{deg}$ and $\mathcal O_{ndeg}$ respectively and let $f = \delta b$ be the associated homogeneous $2$-coboundary. Then
\[
f_0 = f_1 = f_3 = \alpha, \quad f_2 = 2\beta-\alpha, \quad f_+ = f_- = \beta.
\]
\end{lemma}

One very familiar homogeneous $H$-cocycle on $S^1$ of degree $2$ is the orientation cocycle ${\rm Or}$, which assigns the value $+1$, respectively $-1$, to positively oriented, resp. negatively oriented non-degenerate triples, and $0$ to 
all degenerate triples. By the previous lemma, none of its multiples is a coboundary, since the value on positively and negatively oriented triples is not the same. It thus defines a class $[{\rm Or}]$ of infinite order in $ H^2_{(b)}(H\curvearrowright S^1;\Z)$. We now describe general homogeneous $2$-cocycles:
\begin{lemma}\label{lemma f cocycle} Let $f:(S^1)^3\rightarrow \mathbb{R}$ be an invariant homogeneous $H$-cochain. Then $f$ is a cocycle if and only if
\[
f_0 = f_1 = f_3, \quad f_++f_-  = f_2+f_3.
\]
Moreover, $H^2_{(b)}(H\curvearrowright S^1;\Z) \cong \Z$ via the map $[f] \mapsto f_+ - f_-$.
\end{lemma}
\begin{proof} Let $(x,y,z)$ be a positively oriented triple. Writing out the cocycle relations $\delta f(y,x,x,x)=\delta f(x,x,x,y)=\delta f(x,y,x,z)=0$ yields
\[\begin{array}{l}
f(x,x,x)=f(y,x,x)=f(x,x,y),\\
f(y,x,z)-f(x,y,x)=f(x,x,z)-f(x,y,z),
\end{array}\]
which implies that every $2$-cocycle satisfies the $3$ identities of the lemma. The space $\mathcal C^2$ of all cochains satisfying these $3$ identities can be identified with $\Z^3$ via the map $f \mapsto (f_0, f_+, f_-)$. Under this identification 
the space of coboundaries corresponds to $\{(m,n,n)\mid m,n\in \Z\}$, hence the quotient of $\mathcal C^2$ modulo coboundaries is isomorphic to $\Z$ via the map $[f]\mapsto f_+ - f_-$. If there were any other identities satisfied by all $2$-cocycles than those following from the 3 identities above, then $H^2_{(b)}(H\curvearrowright S^1;\Z)$ would be a proper quotient of $\Z$, hence finite, contradicting the fact that $[{\rm Or}]$ has infinite order.
\end{proof}
It follows from the lemma that the class of the orientation cocycle generates a subgroup of index $2$ in $ H^2_{(b)}(H\curvearrowright S^1;\Z)$ and that the generator $-\frac 1 2 [{\rm Or}]$ is represented by the cocycle $c$ satisfying
\begin{equation}\label{EulerFormula}
c_0=c_1=c_3 = c_+ = 0,\quad c_2 = c_- = 1.\end{equation}
\begin{defn} \label{def: Euler cocycle} The homogeneous $2$-cocycle $c\in C^2_b(H\curvearrowright S^1; \Z)$ given by \eqref{EulerFormula} is called the \emph{Euler cocycle}.
\end{defn}
In order to relate the Euler cocycle to the bounded Euler class we need the following computation (see ~\cite[Lemma 2.1]{iozzi}):
\begin{lemma}\label{lemma csigma and c}  If $c \in \mathcal C^2_b(H\curvearrowright S^1; \Z)$ is the Euler cocycle from Definition \ref{def: Euler cocycle} and $c_\sigma \in C^2_b(H;\mathbb{Z})$ denotes the obstruction cocycle associated with the special section  $\sigma: H \to \widetilde{H}$ with $E_\sigma=[0,1)$, then
\[c_\sigma(h_1,h_2)= c([0],h_1\cdot[0],h_1h_2\cdot[0]).\]
Moreover,
\[\mathrm{Or}= - 2 c+\delta b,\]
where $b$ is the $H$-invariant $1$-cochain which takes values $0$ and $1$ on $\mathcal O_{deg}$ and $\mathcal O_{ndeg}$ respectively.
\end{lemma}
\begin{proof} It follows from the explicit definition of $c$ that
\begin{equation*}
c([0],h_1\cdot[0],h_1h_2\cdot[0]) =\left\{\begin{array}{ll} 1 & \mathrm{if \ } 1\leq \sigma(h_1)\sigma(h_2)(0)<\sigma(h_1)(1)<2,  \\0 &\mathrm{if \ }0\leq \sigma(h_1)(0)\leq \sigma(h_1)\sigma(h_2)(0)<1. \end{array}\right.
\end{equation*}
In view of \eqref{csigmaExplicit} this implies $c_\sigma(h_1,h_2)= c([0],h_1\cdot[0],h_1h_2\cdot[0])$. The relation $ \mathrm{Or}=-2 c+\delta b$ is straightforward. 
\end{proof}

From this computation we draw the following conclusion.
\begin{cor}\label{cor: e as cocycle on S1}  The bounded Euler class ${\rm eu}_b$ is representable over the circle. In fact it is represented by the Euler cocycle $c: (S^1)^3 \to \Z$. Similarly, the class $-2\cdot {\rm eu}_b$ is represented over the circle by the orientation cocycle. \qed
\end{cor}
Note that, in particular, for every $x \in S^1$ the homogeneous $2$-cocycle $c_x:H^3\rightarrow \mathbb{Z}$ given by
\[(h_0,h_1,h_2)\mapsto c_x(h_0,h_1,h_2)=c(h_0x,h_1x,h_2x)
\]
represents the bounded Euler class. 

\section{Ghys' Theorem}\label{SecGhysThm}

\subsection{Circle actions with vanishing bounded Euler class}\label{SubsecVanishing}
Before we turn to the proof of Ghys' Theorem in the general case we provide a characterization of circle actions with vanishing bounded Euler class. This characterization can be seen as a special case of Ghys' Theorem, but it is also of independent interest and has a particularly simple proof. Parts of this special case will also be used in the proof of the general theorem.

Recall that the Euler class ${\rm eu}$ was defined as an obstruction class. It thus follows from Lemma \ref{Lift} that if $\rho: \Gamma \to H$ is a circle action, then
\[
\rho^*{\rm eu} = 0 \;\Leftrightarrow \; \text{the action lifts to an action on the real line.}
\]
The following result shows that the vanishing of the \emph{bounded} Euler class has much more drastical consequences:
\begin{prop}\label{VanishingBounded1} Let $\rho: \G \to H$ be a circle action with $\rho^*{\rm eu}_b = 0$. Then the action lifts to an action on the real line which moreover has a fixed point.
\end{prop}
\begin{proof} By assumption there exists a bounded function $u: \G \to \Z$ with $\rho^*c_\sigma = du$, where $c_\sigma$ is the cocycle representing ${\rm eu}_b$ explicitly given in Equations \eqref{FirstFromulacsigma} and \eqref{csigmaExplicit}. By Proposition \ref{LiftExplicit} we have a homomorphism
\[\widetilde{\rho}: \G \to \widetilde{H}, \quad  \widetilde{\rho}(\gamma ) = \sigma(\rho(\gamma )) \cdot i(-u(\gamma )).\]
In particular,
\[
\widetilde{\rho}(\gamma )(0) =  \sigma(\rho(\gamma ))(0) - u(\gamma ).
\]
Now, since $\sigma$ is a bounded section and $u$ is bounded, $\widetilde{\rho}(\gamma )(0)$ is  also  bounded. It follows that
\[F^+(\widetilde{\rho}) := \sup_{\gamma  \in \Gamma }\widetilde{\rho}(\gamma )(0)\]
is well-defined, and it is clearly a fixed point for $\widetilde{\rho}(\G)$.
\end{proof}
Using the second characterization of the bounded Euler class via the translation number we obtain a converse to this result, leading to the following characterization:
\begin{cor}[Circle actions with vanishing bounded Euler class]\label{zeroprop} Let $\rho: \G \to {\rm Homeo}^+(S^1)$ be a circle action. Then the following are equivalent:
\begin{enumerate}[(i)]
\item $\rho^*{\rm eu}_b = 0$.
\item The circle action $\rho$ lifts to an action on the real line which moreover has a fixed point.
\item $\rho(\Gamma)$ fixes a point in $S^1$.
\item $\rho$ is semi-conjugate to the trivial circle action.
\end{enumerate}
\end{cor}
\begin{proof} We have already seen that (i)$\Rightarrow$(ii)  in Proposition \ref{VanishingBounded1}. Conversely, if (ii) holds for a lift $\widetilde{\rho}: \G \to \widetilde{H}$ with fixed point $x_0$, then by Proposition \ref{TNumber}
\[
\rho^*{\rm eu}_b = -\widetilde{\rho}^*[dT_\Z] =-[d\widetilde{\rho}^*T_\Z].
\]
However we have for every $\gamma \in \G$,
\[
\widetilde{\rho}^*T(\gamma ) =  \lim_{n \to \infty}\frac{\widetilde{\rho}(\gamma )^n(x_0)-x_0}{n}=0,
\]
whence $\widetilde{\rho}^*T_\Z$ is bounded and thus (i) holds. The implication (ii)$\Rightarrow$(iii) is obvious, since the projection of a fixed point of a lift is a fixed point for the original action. Conversely, if $\rho(\Gamma)$ fixes $[x_0] \in S^1$, then it acts on $S^1\setminus\{[x_0]\}$ and this action can be lifted to an action on $(x_0, x_0+1)$ and periodically to an action on $\R$ fixing all points in $x_0+\Z$. This shows (ii)$\Leftrightarrow$(iii) and the 
equivalence (iii) $\Leftrightarrow$ (iv) follows from Proposition \ref{semiconj:rem}.
\end{proof} 
Although Corollary \ref{zeroprop} is only a very simple special case of Ghys' Theorem, it is sufficient for many applications. E.g. most of the applications of Ghys' Theorem in higher Teichm\"uller theory depend only on Corollary \ref{zeroprop} (see e.g. \cite{BIW1, BBHIW}). We therefore find it important to point out the above simple proof. Note that a slightly stronger version of Corollary \ref{zeroprop} is established in the appendix.

\subsection{A refined statement of Ghys' Theorem}
We will now prove Ghys' Theorem \ref{IntroMain} (with our Definition \ref{semi:def} of semi-conjugacy), thus establishing that the bounded Euler class is a complete invariant of semi-conjugacy. We will actually prove the following more precise version:

\begin{thm}\label{GhysExplicit}
Let $\rho_1,\rho_2$ be circle actions of $\G$.
\begin{enumerate}[(i)]
\item If $\rho_1^*{\rm eu}_b=\rho_2^*{\rm eu}_b$, then $\rho_1$ and $\rho_2$ are semi-conjugate.
\item If $\rho_1$ and $\rho_2$ are semi-conjugate and either of them has a fixed point, then both have a fixed point and $\rho_1^*{\rm eu}_b = \rho_2^*{\rm eu}_b = 0$.
\item If $\rho_1$ is fixed point free and left-semi-conjugate to $\rho_2$, then $\rho_1^*{\rm eu}_b=\rho_2^*{\rm eu}_b \neq 0$.
\end{enumerate} 
\end{thm}
Note that in the situation of (iii), $\rho_1$ and $\rho_2$ are actually semi-conjugate by (i). This proves the following result alluded to in the introduction, also proven in \cite[Proposition 1.4]{Matsumoto}.
\begin{cor}\label{mistake:small}
If a fixed point free circle action  $\rho_1$ is left-semi-conjugate to a circle action $\rho_2$, then they are semi-conjugate. In particular, left-semi-conjugacy defines an equivalence relation on the set of all fixed point free circle actions.
\end{cor}
Part (ii) of Theorem \ref{GhysExplicit} follows directly from Corollary \ref{zeroprop}: If, say, $\rho_1$ has a fixed point, then it is semi-conjugate to the trivial circle action by the implication (iii) $\Rightarrow$ (iv), whence also $\rho_2$ is semi-conjugate to the trivial circle action and thus has a fixed point by the implication (iv) $\Rightarrow$ (iii). Then, by the implication 
(iii) $\Rightarrow$ (i) we have $\rho_1^*{\rm eu}_b = \rho_2^*{\rm eu}_b = 0$. Thus it remains to show only (i) and (iii), which we will do in the next two subsections.

\subsection{Same bounded Euler class implies semi-conjugacy}
In this subsection we are going to establish Part (i) of Theorem  \ref{GhysExplicit}. Our proof is a slight variation of Ghys' original proof, which emphasizes the similarity to the proof of Proposition \ref{VanishingBounded1}. 

To fix notation, let $\rho_1,\rho_2$ be circle actions with the same bounded Euler class $\rho_1^*{\rm eu}_b=\rho_2^*{\rm eu}_b$. We claim that $\rho_1$ and $\rho_2$ are semi-conjugate. By symmetry it suffices to show that $\rho_1$ is left-semi-conjugate to $\rho_2$.

Let $\widetilde{\Gamma}$ be the central extension of $\Gamma$ which corresponds to $\rho_1^*{\rm eu} = \rho_2^*{\rm eu}$. Then we can choose lifts $\widetilde{\rho_1}, \widetilde{\rho_2}$ making the diagram
\[\begin{xy}
\xymatrix{
0\ar[r]&\Z\ar[r]^i&\widetilde{H} \ar[r]^p&H \ar[r]&1\\
0\ar[r]&\Z\ar[u]_{\widetilde{\rho_j}}\ar[r]^i&\widetilde{\G}\ar[u]_{\widetilde{\rho_j}} \ar[r]&\G \ar[u]_{\rho_j} \ar[r]&1
}
\end{xy}\]
commute. Since $\rho_1^*{\rm eu}_b=\rho_2^*{\rm eu}_b$ and the diagrams commute we have
\[
[d\widetilde{\rho_1}^*T_\Z] = \widetilde{\rho_1}^*[dT_\Z] = -\widetilde{\rho_1}^*(p^*{\rm eu}_b) = -\widetilde{\rho_2}^*(p^*{\rm eu}_b) = \widetilde{\rho_2}^*[dT_\Z] = [d\widetilde{\rho_2}^*T_\Z].
\]
This implies that there exist a homomorphism $u: \widetilde{\Gamma} \to \Z$ and a bounded function $b: \widetilde{\Gamma} \to \Z$ such that
$\widetilde{\rho_1}^*T_\Z - \widetilde{\rho_2}^*T_\Z = u +b$. It follows that $\widetilde{\rho_1}^*T - \widetilde{\rho_2}^*T - u$ is a bounded homogeneous function, hence $0$. Thus,
\[
\widetilde{\rho_1}^*T - \widetilde{\rho_2}^*T = u.
\]
Replacing the lift $\widetilde{\rho_2}$ by $\widetilde{\rho_2}+i\circ u$ we can ensure that $u = 0$. Assume that $\widetilde{\rho_2}$ is chosen in that way. Then
for every $g \in \widetilde{H}$,
\[
|T(\widetilde{\rho}_1(g)^{-1}\widetilde{\rho}_2(g))|\leq |-T(\widetilde{\rho}_1(g))+T(\widetilde{\rho}_2(g))|+ D(T) = D(T),
\]
where $D(T)$ is the defect of the quasimorphism $T$. In particular, $\widetilde{\rho}_1(g)^{-1}\widetilde{\rho}_2(g)$ has uniformly bounded translation number and thus
\[
\widetilde{\varphi}(x):=\sup_{g\in \widetilde{\G}} (\widetilde{\rho}_1(g)^{-1}\widetilde{\rho}_2(g)(x))
\]
is well-defined. By definition we have for every $g_0 \in \widetilde{\G}$,
\begin{align*}
\widetilde{\varphi}(\widetilde{\rho}_2({g}_0)(x)) &=
\sup_{{g}\in{\widetilde{\G}}} \widetilde{\rho}_1({g})^{-1}(\widetilde{\rho}_2({g})(\widetilde{\rho}_2({g}_0)(x)))\\
& = \sup_{{g}\in{\widetilde{\G}}} \widetilde{\rho}_1({g}\, {g}_0^{-1})^{-1}(\widetilde{\rho}_2({g})(x))\\
& = \widetilde{\rho}_1({g}_0)\left( \sup_{{g}\in{\widetilde{\G}}} \widetilde{\rho}_1({g})^{-1}(\widetilde{\rho}_2({g})(x))\right)\\
& = \widetilde{\rho}_1({g}_0)(\widetilde {\varphi}(x)).
\end{align*}
Moreover, being the supremum of increasing maps which commute with integral translations, the map $\widetilde{\varphi}\colon\R\to\R$ is non-decreasing and commutes with integral translations, so it is a good lift of a non-decreasing degree one map $\varphi\colon S^1\to S^1$.  It follows that $\varphi$ realizes the desired left-semi-conjugacy from $\rho_1$ to $\rho_2$. This finishes the proof of Part (i) of Theorem \ref{GhysExplicit}.

\begin{rem} \label{rem Ghys 2} Note that it now immediately follows that Condition (i) of Theorem \ref{ThmEquivalence} implies Ghys' condition stated as Condition (vi) in Remark \ref{rem Ghys 1}. Indeed, if the bounded Euler classes are equal, then so are the (unbounded) Euler classes and the map $\widetilde {\varphi}$ in the above proof gives the map required in Condition (vi).\end{rem}

\subsection{Semi-conjugacy implies same bounded Euler class}\label{SubsectionGhysiii}
In this subsection we establish the remaining Part (iii) of Theorem \ref{GhysExplicit} thereby finishing the proof of the theorem. Here we will finally make use of the third (geometric) characterization of the bounded Euler class. 

Instead of Theorem \ref{GhysExplicit}.(iii) we will actually prove a slightly stronger statement. To state this result we introduce the following notation. Throughout this section we will fix two circle actions $\rho_1, \rho_2$ of $\Gamma$ 
and a semi-conjugacy $\phi$ from $\rho_1$ to $\rho_2$. We will \emph{not} assume a priori that $\rho_1$ is fixed point free. For each $\gamma \in \Gamma$ we fix lifts $\widetilde{\rho_1}(\gamma)$ and $\widetilde{\rho_2}(\gamma)$ of $\rho_1(\gamma)$ respectively $\rho_2(\gamma)$. Suppose now that $\widetilde{\phi}$ is some good lift of $\phi$. 
Since $\widetilde{\rho_1}(\gamma)\widetilde{\phi}$ and $\widetilde{\phi} \widetilde{\rho_2}(\gamma)$ are lifts of the same map and are invariant under integral translations, there exists 
a map $n_\gamma: \R \to \Z$ (dependent on $\widetilde{\phi}$), invariant under integral translations, such that for all $x \in \R$,
\begin{equation}\label{ngamma}
\widetilde{\rho_1}(\gamma)\widetilde{\phi}(x) = \widetilde{\phi}(\widetilde{\rho_2}(\gamma)(x)) + n_{\gamma}(x).
\end{equation}

\begin{prop}\label{MainProposition} Let $\rho_1, \rho_2$ be circle-actions of $\Gamma$ and let $\phi$ be a semi-conjugacy from $\rho_1$ to $\rho_2$. Let a good lift $\widetilde{\phi}$ of $\phi$ be fixed and let $n_\gamma: \R \to \Z$ be defined by \eqref{ngamma}. Consider the following statements:
\begin{enumerate}[(1)]
\item $\rho_1(\Gamma)$ does not have a global fixed point in $S^1$.
\item $\phi$ is not the constant map.
\item There exists a good lift $\widetilde{\phi}$ of $\phi$ such that for each $\gamma \in \Gamma$ the map $n_\gamma$ given by \eqref{ngamma} is constant.
\item There exists a good lift $\widetilde{\phi}$ of $\phi$ such that $\widetilde{\rho_1}(\gamma)\widetilde{\phi}(x) = \widetilde{\phi}(\widetilde{\rho_2}(\gamma)(x))$ for all $\gamma \in \Gamma$ and $x \in \R$.
\item There exists a non-empty $\rho_2(\Gamma)$-invariant subset $K\subset S^1$ such that $\phi|_{K}$ is injective.
\item $\rho_1^*{\rm eu}_b=\rho_2^*{\rm eu}_b$.
\end{enumerate}
Then the implications (1)$\Rightarrow$(2)$\Rightarrow$(3)$\Rightarrow$(4)$\Rightarrow$(5)$\Rightarrow$(6) hold.
\end{prop}
Note that the implication (1)$\Rightarrow$(6) gives Part (iii) of Theorem \ref{GhysExplicit}. 
\begin{proof}[Proof of Proposition \ref{MainProposition}] The implication (1)$\Rightarrow$(2) is obvious, so we turn directly to the proofs of the implications (2)$\Rightarrow$(3)$\Rightarrow$(4)$\Rightarrow$(5)$\Rightarrow$(6). 

Assume that (2) holds and fix $\gamma \in \Gamma$. Let $\widetilde{\phi}$ be a good lift of $\phi$. Since $\phi$ is non-constant we find distinct elements $a_0, b_0 \in \R$ with $b_0-a_0 \in (0,1)$ and $\widetilde{\phi}(b_0)-\widetilde{\phi}(a_0) \in (0,1)$.
Since $\widetilde{\rho_1}(\gamma)$ is strictly increasing and commutes with integral translations, this implies at once that 
\begin{equation}\label{equProp4.5}
0<\widetilde{\rho_1}(\gamma)(\widetilde{\phi}(b_0))-\widetilde{\rho_1}(\gamma)(\widetilde{\phi}(a_0))<1.
\end{equation} 
On the other hand, since $\widetilde{\phi}\circ \widetilde{\rho_2}(\gamma)$ is non-decreasing and commutes with integral translations, we also have 
$0\leq \widetilde{\phi}(\widetilde{\rho_2}(\gamma)(b_0))-\widetilde{\phi}(\widetilde{\rho_2}(\gamma)(a_0))\leq 1$. However, these inequalities must both be strict, because otherwise we would have
\[
\rho_1(\gamma)(\phi([b_0]))= \phi(\rho_2(\gamma)([b_0]))=\phi(\rho_2(\gamma)([a_0]))=\rho_1(\gamma)(\phi([a_0])),
\]
which contradicts (\ref{equProp4.5}).  We have thus shown that
\[
0<\widetilde{\rho_1}(\gamma)(\widetilde{\phi}(b_0))-\widetilde{\rho_1}(\gamma)(\widetilde{\phi}(a_0))<1\, ,\quad 
0<\widetilde{\phi}(\widetilde{\rho_2}(\gamma)(b_0))-\widetilde{\phi}(\widetilde{\rho_2}(\gamma)(a_0))<1 \ .
\]
Subtracting the second inequality from the first we deduce that
$n_\gamma(b_0)-n_\gamma(a_0)\in [0,1)-[0,1)=(-1,1)$. Since both are integers we deduce that $n_\gamma(b_0)=n_\gamma(a_0)$, which implies that
$n_\gamma$ is constant on $E := (a_0+\Z)\cup(b_0+\Z)$. 

Now let $x \in \R \setminus E$. Then the interval $(x, x+1)$ contains one translate of $a_0$ and one translate of $b_0$, and these take different values under $\widetilde{\phi}$. We thus find $e \in E$ with
$x<e<x+1$ and $\widetilde{\phi}(x) \neq \widetilde{\phi}(e)$, whence  $\{x-e, \widetilde{\phi}(x) - \widetilde{\phi}(e)\} \subset [0,1)$ and $n_\gamma(x) - n_\gamma(e)\in (-1,1)$, 
so that $n_\gamma(x)=n_\gamma(e)$. This finishes the proof of the implication (2)$\Rightarrow$(3).

Now assume that (3) holds, i.e. for every $\gamma \in \Gamma$ we have $n_\gamma(x) = n_\gamma$ for some constant $n_\gamma$. We can then replace the lift $\widetilde{\rho_1}(\gamma)$ by the lift $\widetilde{\rho_1}(\gamma) - n_\gamma$ and thereby achieve that for all $x \in \R$,
\begin{equation}\label{ngammaconst}
\widetilde{\rho_1}(\gamma)\widetilde{\phi}(x) = \widetilde{\phi}(\widetilde{\rho_2}(\gamma)(x)),\end{equation}
which is (4). 

We now deduce (5) from (4). Given $x_0 \in \R$ we define 
\[
S_{x_0} = \{x \in \R\mid \widetilde{\phi}(x) = \widetilde{\phi}(x_0)\} = \widetilde{\phi}^{-1}(\widetilde{\phi}(x_0)).
\]
Since $\widetilde{\phi}$ is increasing, the sets $S_{x_0}$ are connected, and since $\widetilde{\phi}$ commutes with integral translations we have $S_{x_0} \subset (x_0-1, x_0+1)$. In particular, each $S_{x_0}$ is bounded and if we define $\alpha(x_0):= \inf S_{x_0}$ and $\beta(x_0) := \sup S_{x_0}$, then
\[
S_{x_0}  \in \{(\alpha(x_0), \beta(x_0)), (\alpha(x_0), \beta(x_0)], [\alpha(x_0), \beta(x_0)), [\alpha(x_0), \beta(x_0)]\},
\]
is an open, half-closed or closed interval. Since $\mathbb{R}$ is connected, not all of these intervals can be open.
Thus the sets 
\[
\widetilde{K}_- := \{x \in \R\mid x = \inf S_x\} \quad \text{ and }\quad \widetilde{K}_+ := \{x \in \R\mid x = \sup S_x\}
\]
cannot both be empty (though it is easy to construct examples where one of them is empty).

We observe that the restriction $\widetilde{\phi}|_{\widetilde{K}_\pm}$ are both injective. Assume first that $x_1, x_2 \in \widetilde{K}_-$ and $\widetilde{\phi}(x_1) = \widetilde{\phi}(x_2)$. Then $S_{x_1} = S_{x_2}$ and thus
\[
x_1 = \inf S_{x_1} = \inf S_{x_2} = x_2,
\]
showing that  $\widetilde{\phi}|_{\widetilde{K}_-}$ is injective. Replacing $\inf$ by $\sup$, we deduce similarly that  $\widetilde{\phi}|_{\widetilde{K}_+}$ is injective.



Now we claim that $\widetilde{K}_\pm$ are invariant under $\widetilde{\rho_2}(\gamma)$ for every $\gamma \in \Gamma$. For this it suffices to check that $\widetilde{\rho_2}(\gamma)(S_x)=S_{\widetilde{\rho_2}(\gamma)x}$ for every $x\in\R$, $\gamma\in\Gamma$. This follows from the chain of equivalences
\begin{eqnarray*}
 y\in\widetilde{\rho_2}(\gamma)(S_x) & \Longleftrightarrow &  \widetilde{\rho_2}(\gamma^{-1})(y)\in S_x  \quad \Longleftrightarrow \quad  \widetilde{\varphi}(\widetilde{\rho_2}(\gamma^{-1})y)=\widetilde{\varphi}(x) \\
&\Longleftrightarrow&   \widetilde{\rho_1}(\gamma^{-1})\widetilde{\varphi}(y)=\widetilde{\varphi}(x) \quad  \Longleftrightarrow  \quad \widetilde{\varphi}(y)=\widetilde{\rho_1}(\gamma)\widetilde{\varphi}(x)=\widetilde{\varphi}(\widetilde{\rho_2}(\gamma)(x))\\
&\Longleftrightarrow& y\in S_{\widetilde{\rho_2}(\gamma)x}.
\end{eqnarray*}

Now let $K_\pm$ be the projections of $\widetilde{K}_{\pm}$ on $S^1$. Then $K_\pm$ are $\rho_2(\Gamma)$-invariant and $\phi$ is injective on both $K_+$ and $K_-$. Since at least one of these two sets is non-empty, this finishes the proof of the implication (4)$\Rightarrow$(5). 

Finally, we establish the implication (5)$\Rightarrow$(6): Let $K$ be as in (4) and let $x \in K$. By Lemma \ref{lemma f cocycle} the cohomology class $\rho_2^*{\rm eu}_b$ is represented by the cocycle 
\[
\rho_2^*c_x(g_0, g_1, g_2) = c(\rho_2(g_0)x, \rho_2(g_1)x, \rho_2(g_2)x).
\]
Note that for $j=0,1,2$ the points $\rho_2(g_j)x$ are all contained in $K$, since $K$ is $\rho_2(\Gamma)$-invariant. It thus follows from injectivity of $\phi$ on $K$ that they are pairwise distinct if and only if their images under $\phi$ are pairwise distinct. 
Since $\phi$ also preserves their weak orientation, we deduce that the triples $(\rho_2(g_0)x, \rho_2(g_1)x, \rho_3(g_2)x)$ and $(\widetilde{\phi}(\rho_2(g_0)x), \widetilde{\phi}(\rho_2(g_1)x), \widetilde{\phi}(\rho_2(g_2)x))$ are in the same $H$-orbit. Indeed, this follows from the classification of $H$-orbits on $(S^1)^3$ in Subsection \ref{explicit cocycles}. Since $c$ is $H$-invariant we obtain
\begin{eqnarray*}
\rho_2^*c_x(g_0, g_1, g_2) &=& c(\phi(\rho_2(g_0)x), \phi(\rho_2(g_1)x), \phi(\rho_2(g_2)x))\\
&=& c(\rho_1(g_0)\phi(x), \rho_1(g_1)\phi(x), \rho_1(g_2)\phi(x))\\
&=& \rho_1^*c_{\phi(x)}(g_0, g_1, g_2).
\end{eqnarray*}
Since the cocycle $\rho_1^*c_{\phi(x)}$ represents $\rho_1^*{\rm eu}_b$, we deduce that $\rho_1^*{\rm eu}_b = \rho_2^*{\rm eu}_b$. This finishes the proof.
\end{proof}
At this point we have finished the proof of Theorem \ref{GhysExplicit} and thereby of Theorem \ref{IntroMain}.
\begin{rem} In \cite[Equation (1), Proof of Proposition 5.2]{Ghys0}) our map $n_\gamma$ is denoted by $u(\gamma)$. It is assumed to be constant independently of whether $\phi$ is constant or not. The following example shows that this is not true in general. Let $\rho_1$ be the trivial circle action of $\Z$ and $\rho_2$ be the circle action sending $1$ to the rotation by $1/2$. Then $\rho_1$ is left semi-conjugate to $\rho_2$ by 
Proposition \ref{semiconj:rem} (ii). The left semi-conjugacy can be given by the constant map $\varphi(x)\equiv [0]$ which lifts to $\widetilde{\varphi}:x\mapsto \left\lfloor x \right\rfloor$. A lift of $\rho_1(1)$ is the identity 
and 
a lift of $\rho_2(1)$ is the translation $T_{1/2}$ by $1/2$. Then $\rho_1(1)\varphi=\varphi \rho_2(1)$ on the circle but the translation $x\mapsto\widetilde{\varphi}(x)-\widetilde{\varphi}(T_{1/2}(x))=  \left\lfloor x \right\rfloor- \left\lfloor x +1/2\right\rfloor$ depends on $x$ since it is $0$ for $x\in [0,1/2)+\Z$ and  $-1$ for $x\in [1/2,1)+\Z$. More generally, neither of the statements (2)--(5) 
is correct without the assumption that $\rho_1$ is fixed point free. For example, 
if $\rho_1$ has a fixed point then we can alway choose $\phi$ to be constant. In that case, every set $K \subset S^1$ on which $\phi$ is injective is a singleton. If ${\rho}_2(\Gamma)$ is fixed point free, then such a set cannot be invariant. The reader may check that in this case our set $K_2$ constructed in the proof is indeed a singleton, and that the proof of invariance breaks down in the absence of (3), e.g. in the situation of the example above.
\end{rem}

\subsection{The minimal case: Semi-conjugacy equals conjugacy}\label{Section Min Actions}
Recall that a circle action $\rho: \G \to \omeo$ is \emph{minimal} if every $\rho(\G)$-orbit is dense in $S^1$. The following proposition shows that for minimal circle actions, the notions of conjugacy and semi-conjugacy coincide. This implies in particular that Theorem \ref{GhysMinimal} follows from Theorem \ref{IntroMain}.
\begin{prop}[Ghys]\label{dense:orbits}
Let $\rho_1,\rho_2\colon \G\to \omeo$ be minimal circle actions. Then the following are equivalent:
\begin{enumerate}[(i)]
\item $\rho_1$ is left-semi-conjugate to $\rho_2$.
\item $\rho_1$ and $\rho_2$ are semi-conjugate.
\item $\rho_1$ and $\rho_2$ are conjugate.
\end{enumerate}
\end{prop}
\begin{proof} Since minimal actions are fixed point free, the equivalence (i)$\Leftrightarrow$(ii) follows from Corollary \ref{mistake:small}. Moreover, the implication (iii)$\Rightarrow$(i) holds trivially. Concerning the implication (i)$\Rightarrow$(iii) assume that $\rho_1$ is left-semi-conjugate to $\rho_2$ via $\phi$. Then the image of $\varphi$ is $\rho_1(\G)$-invariant, whence dense in $S^1$ by minimality. 
This in turn implies that the image of $\widetilde{\varphi}$ is dense in $\R$. So the map $\widetilde{\varphi}$, being non-decreasing and commuting with integral translations, 
is continuous and surjective. Therefore, the same is true for $\varphi$, and we are left to show that $\varphi$ is also injective.

Suppose by contradiction that there exist distinct points $x,y\in S^1$ such that $\varphi(x)=\varphi(y)$, and choose lifts $\widetilde{x},\widetilde{y}$ of $x,y$ in $\R$ such that $\widetilde{x}<\widetilde{y}<\widetilde{x}+1$. Since $\widetilde{\varphi}$ is non-decreasing and commutes with integral translations, 
we have either $\widetilde{\varphi}(\widetilde{y})=\widetilde{\varphi}(\widetilde{x})$ or $\widetilde{\varphi}(\widetilde{y})=\widetilde{\varphi}(\widetilde{x}+1)$. In any case, $\widetilde{\varphi}$ is constant on a non-trivial interval, so there 
exists an open subset $U\subseteq S^1$ such that $\varphi|_U$ is constant. 
Let now $x$ be an arbitrary point of $S^1$. By minimality of $\rho_2$ there exists $\gamma \in\G$ such that $\rho_2(\gamma )^{-1}(x)\in U$, and consequently $V:=\rho_2(\gamma )(U)$ is an open neighborhood of $x$. Now
\[
\varphi|_V=(\varphi\rho_2(\gamma ))|_U\circ \rho_2(\gamma )^{-1}|_V=
(\rho_1(\gamma )\varphi)|_U\circ \rho_2(\gamma )^{-1}|_V,
\]
whence $\varphi$ is locally constant. It follows that $\varphi$ is constant, and this contradicts the fact that $\varphi$ is surjective.
\end{proof}
We have now established Theorems \ref{GhysMinimal} and \ref{IntroMain} mentioned in the introduction.

\section{Variations and examples}\label{SecVariations}

\subsection{Circle actions of $\Z$ and the rotation number}\label{SecClassical}
Let us spell out a few special immediate consequences of Ghys' Theorem. We start with the case where $\G = \Z$. In this case a circle action $\rho: \G \to \omeo$ is given by a single invertible transformation $\rho(1) \in \omeo$. The action lifts to $\widetilde{\rho}: \Z \to \widetilde{H}$ and following Poincar\'e we define its rotation number as
\[
R(\rho) := T(\widetilde{\rho}(1)) \mod\Z,
\]
where $T$ is the real valued translation number defined in Section \ref{SubsecEuler2}. 
\begin{example}\label{Rotation}
Given $\alpha \in \R/\Z$ we denote by $R_\alpha \in \omeo$ the rotation by $\alpha$. Then the $\Z$-action $\rho$ with $\rho(1) = R_\alpha$ has rotation number $\alpha$. In particular, every rotation number can be realized by a rotation.
\end{example}

The fact that any $\Z$-action lifts is illustrated by $\rho^*(\eu)=0\in H^2(\Z;\Z)=\{0\}$. Thus, the unbounded Euler class cannot give any information for $\Z$-actions. The case of the bounded Euler class is much more interesting: 

\begin{cor}[Poincar\'e]\label{Poincare}
For circle actions $\rho_1, \rho_2: \Z \to \omeo$ the following are equivalent:
\begin{enumerate}[(i)]
\item $\rho_1$ and $\rho_2$ are semi-conjugate.
\item $\rho_1^*{\rm eu}_b = \rho_2^*{\rm eu}_b$.
\item $R(\rho_1) = R(\rho_2)$.
\end{enumerate}
In particular, Poincar\'e's rotation number is a complete semi-conjugacy invariant for circle actions of $\Z$ (and a complete conjugacy invariant for minimal $\Z$-actions).
\end{cor}
\begin{proof} The equivalence (i)$\Leftrightarrow$(ii) is a special case of Theorem \ref{IntroMain}. For $j=1,2$ we have
\[
\rho_j^*{\rm eu}_b = \widetilde{\rho_j}^*p^*{\rm eu}_b =  -\widetilde{\rho_j}^*[dT_\Z] =-[d \widetilde{\rho_j}^*T_\Z], 
\]
whence (ii) is equivalent to $[d(\widetilde{\rho_1}^*T_\Z - \widetilde{\rho_2}^*T_\Z )] = 0$. This in turn means that there exists a homomorphism $f \in {\rm Hom}(\Z, \Z)$ such that the quasimorphism $\widetilde{\rho_1}^*T_\Z - \widetilde{\rho_2}^*T_\Z - f$ is bounded. Now using the fact that a quasimorphism is bounded if and only if its homogeneization is trivial we see that the latter condition is equivalent to
\[
\widetilde{\rho_1}^*T - \widetilde{\rho_2}^*T  = f \in{\rm Hom}(\Z, \Z).
\]
Since two homogeneous functions on $\Z$ agree iff they agree on $1$ we see that
this condition is equivalent to
\[
\widetilde{\rho_1}^*T(1) - \widetilde{\rho_2}^*T(1) =T(\widetilde{\rho_1}(1))-T(\widetilde{\rho_2}(1))\in \Z,
\]
i.e. $R(\rho_1) = R(\rho_2)$.
\end{proof}

\subsection{The Hirsch-Thurston theorem}
Let us denote by ${\rm Rot}(S^1) \cong \R/\Z$ the subgroup of $\omeo$ given by rotations. A circle actions which factors through ${\rm Rot}(S^1)$ will be called a \emph{rotation action}. It follows from Example \ref{Rotation} and Corollary \ref{Poincare} that every $\Z$-action is semi-conjugate to a rotation action. This is more generally true for actions of amenable groups; 
the corresponding result is usually attributed to Hirsch and Thurston (see e.g. \cite{Calegari_Book}), since it can be derived easily from results in \cite{HirschThurston}.

\begin{cor}[Hirsch-Thurston]\label{AmenableGroups} Every circle action $\rho: \G \to \omeo$ of an amenable group is semi-conjugate to a rotation action.
\end{cor}
\begin{proof} By a classical result of Trauber (see e.g.~\cite{Gromov, Ivanov}) the bounded cohomology of $\G$ with \emph{real} coefficient vanishes. Thus the connecting homomorphism
\[
\delta: H^1(\Gamma; \R/\Z) \to H^2_b(\Gamma; \Z)
\]
of the Gersten exact sequence (see \cite[Prop. 8.2.12]{Monod}) is an isomorphism. Let $\alpha := \rho^*{\rm eu}_b \in H^2_b(\Gamma; \Z)$ and $\beta := \delta^{-1}(\alpha)$. Then under the isomorphism $H^1(\Gamma; \R/\Z) \cong {\rm Hom}(\Gamma, \R/\Z) =  {\rm Hom}(\Gamma, {\rm Rot}(S^1))$ the class $\beta$ corresponds to a homomorphism $\rho': \G \to  {\rm Rot}(S^1)$. 
Now a standard diagram chase shows that $(\rho')^*{\rm eu}_b = \delta(\beta) = \rho^*{\rm eu}_b$, whence $\rho$ and $\rho'$ are semi-conjugate.
\end{proof}

\subsection{Real bounded Euler class}\label{RBEC}
In many applications, computations in integral bounded cohomology are difficult, and thus one relies on real bounded cohomology. The image of ${\rm eu}_b$ in $H^2_b(H; \R)$ under the change of coefficients map $H^2_b(H; \Z) \to H^2_b(H; \R)$ is called the \emph{real bounded Euler class} and denoted ${\rm eu}_b^\R$. Corollary \ref{VanishingBounded1}
has the following real counterpart:
\begin{cor}\label{VanishingRealEulerClass} Let $\rho: \G \to \omeo$ be a circle action with $\rho^*{\rm eu}_b^\R = 0$. Then $\rho([\G, \G])$ fixes a point on $S^1$.
\end{cor}
\begin{proof} Since $\rho^*{\rm eu}_b^\R = 0$ we can argue as in the proof of Corollary \ref{AmenableGroups} and prove that $\rho$ is semi-conjugate to an action $\rho': \G \to {\rm Rot}(S^1) < \omeo$. In particular, $\rho|_{[\G, \G]}$ is semi-conjugate to $\rho'|_{[\G, \G]}$. Now since ${\rm Rot}(S^1)$ is abelian, $\rho'$ vanishes on $[\G, \G]$. It follows that $(\rho|_{[\G,\G]})^*{\rm eu}_b = (\rho'|_{[\G,\G]})^*{\rm eu}_b = 0$, 
whence $\rho([\G, \G])$ fixes a point on $S^1$ by Corollary \ref{VanishingBounded1}.
\end{proof}

\section{Alternative characterizations of semi-conjugacy}\label{SecEquivalence}

\subsection{Regularity of semi-conjugacies}\label{Regularity}
Having established Theorems \ref{GhysMinimal} and \ref{IntroMain} and some of their consequences, we now return to the characterizations of semi-conjugacy given in Theorem \ref{ThmEquivalence} of the introduction. We start by discussing the issue of regularity of semi-conjugacies. In general, if two circle actions $\rho_1$ and $\rho_2$ are semi-conjugate it does \emph{not} follow that they are semi-conjugate via continuous left-semi-conjugacies.
 A concrete counterexample is given as follows.

\begin{example}\label{example not continuous} Let $\rho_1$ be the action of $\Z$ given by sending the generator $1$ to the rotation by $\pi$. Let $\rho_2$ be an action of $\Z$ with rotation number $\frac{1}{2}$ for which $\rho_2(2)$ has precisely two fixed points. For example, the generator could be sent to 
the fixed point free lift of a parabolic isometry to the double cover of $S^1=\partial \mathbb{H}^2$. Both actions have rotation number $1/2$, so that they are semi-conjugate, say, 
$\rho_1$ is right-semi-conjugate to $\rho_2$ via $\varphi:S^1\rightarrow S^1$. By definition, $\varphi$ sends orbits for the $\rho_1$-action to orbits for the $\rho_2$-action. Now all $\rho_1$ orbits have precisely two points, while only one $\rho_2$ orbit has two points (and the other orbits have infinite order). It follows that the image of $\varphi$ is equal to the unique $\rho_2$ orbit consisting of two points, hence the map $\varphi$ cannot be continuous. 
Even worse, the semi-conjugacy $\varphi':S^1\rightarrow S^1$ in the opposite direction, i.e. from $\rho_1$ to $\rho_2$ cannot be chosen continuous either. Indeed, let $\{x_1, x_2\}$ be the unique $\rho_2$-orbit containing two points. Then
$\phi'$ has to send $x_1$ and $x_2$ to a pair of antipodal points $y,\overline{y}$. Now restrict to the index two subgroup $2\Z<\Z$ and look at the restricted orbits: The restricted $\rho_1$-action is trivial, so orbits for the restricted $\rho_2$-action have to be sent to points. But $x_1$ and $x_2$ are accumulation points of the same restricted $\rho_2$-orbit, which is all mapped to a point $z$.
 Then $z$ cannot be both equal to $y$ and $\overline{y}$, so that $\varphi'$ is not continuous.  
 \end{example}

Things get better if we replace continuity with the less demanding notion of semicontinuity. 
Recall that
a non-decreasing degree one map $\varphi\colon S^1\to S^1$ is called \emph{upper semicontinuous} if it admits an upper semicontinuous good lift $\widetilde{\varphi}\colon\R\to\R$.
Indeed we can show:
\begin{lemma}\label{uppersemi}
If a circle action $\rho_1: \G\to H$ is left-semi-conjugate to a circle action $\rho_2: \G\to H$, then it is left-semi-conjugate to $\rho_2$ via an upper semicontinuous map $\varphi'\colon S^1\to S^1$.
\end{lemma}

\begin{proof} Let $\phi$ be an arbitrary left-semi-conjugacy from $\rho_1$ to $\rho_2$. If $\phi$ is constant then there is nothing to show, hence we may assume that $\phi$ is non-constant. We then define $\widetilde{\phi}$, $\widetilde{\rho_1}$, $\widetilde{\rho_2}$ and $n_\gamma$ as in the beginning of Subsection \ref{SubsectionGhysiii} and also define a new function $\widetilde{\varphi}': \R \to \R$ by
\[
\widetilde{\varphi}'(x):=\sup \{\widetilde{\varphi}(y)\, |\, y<x\}.
\]
Since $\widetilde{\varphi}'$ is non-decreasing and commutes with integral translations, it is the good lift of a non-decreasing degree one map $\varphi'\colon S^1\to S^1$. We claim that $\varphi'$ is a left-semi-conjugacy from $\rho_1$ to $\rho_2$.

In order to prove our claim we fix $\gamma \in \Gamma$ and abbreviate $\widetilde{h_j} := \widetilde{\rho_j}(\gamma) \in \widetilde{H}$ for $j=1,2$. By \eqref{ngamma} we have for every $y \in \R$,
\begin{equation}\label{e4}
\widetilde{h_1} \widetilde{\varphi}(y)=\widetilde{\varphi}(\widetilde{h_2} (y))+n_{\gamma}(y).
\end{equation}
By the implication (2)$\Rightarrow$(3) in Proposition \ref{MainProposition}, there exists an integer $m \in \mathbb N$ such that $n_\gamma \equiv m$. Now  for every $x \in \R$ we have
\begin{eqnarray*}
\widetilde{h_1} \widetilde{\varphi}'(x) &=&
\widetilde{h_1} (\sup \{\widetilde{\varphi}(y)\mid y<x\})=
\sup \{\widetilde{h_1} \widetilde{\varphi}(y)\mid y<x\}\\
&=& \sup\{\widetilde{\varphi}(\widetilde{h_2} (y))+m\mid y<x\}= \sup\{\widetilde{\varphi}(y)\mid y<\widetilde{h_2} (x)\}+m\\
&=& \widetilde{\varphi}'(\widetilde{h_2} (x))+m,
\end{eqnarray*}
which implies that $\rho_1(\gamma)\varphi'=\varphi'\rho_2(\gamma)$, and concludes the proof.
\end{proof}

\subsection{Proof of Theorem \ref{ThmEquivalence}} The equivalences (i)$\Leftrightarrow$(ii)$\Leftrightarrow$(v) of Theorem \ref{ThmEquivalence} are immediate from Theorem \ref{GhysExplicit}. The equivalence (i)$\Leftrightarrow$(iii)$\Leftrightarrow$(iv) of Theorem \ref{ThmEquivalence} follows from the following corollary of Proposition \ref{MainProposition}. 
\begin{cor}\label{Bucher} For circle actions $\rho_1$ and $\rho_2$ the following are equivalent:
\begin{enumerate}[(i)]
\item There exists a left-semi-conjugacy from $\rho_1$ to $\rho_2$ which satisfies Property (4) of Proposition \ref{MainProposition}.
\item There exists a left-semi-conjugacy from $\rho_1$ to $\rho_2$ which satisfies Property (5) of Proposition \ref{MainProposition}.
\item $\rho_1$ and $\rho_2$ are semi-conjugate.
\end{enumerate}
\end{cor}
\begin{proof}
The implications (i)$\Rightarrow$(ii)$\Rightarrow$(iii) of the corollary follow from the implications (4)$\Rightarrow$(5)$\Rightarrow$(6) in Proposition \ref{MainProposition} and Part (i) of Theorem \ref{GhysExplicit}. Conversely assume that (iii) holds and that $\rho_1$ is left-semi-conjugate to $\rho_2$ via $\phi$. If $\phi$ is non-constant then (i) and (ii) hold by the implications (2)$\Rightarrow$(4)$\Rightarrow$(5) of Proposition \ref{MainProposition}. 
Now assume, on the other hand, that $\phi$ is constant. Then the image of $\phi$ is a fixed point $[x_1]$ for $\rho_1$. According to Part (ii) of Theorem \ref{GhysExplicit} there is also a fixed point $[x_2]$ of $\rho_2$. Let $\widetilde{x_1},  \widetilde{x_2} \in \R$ be lifts of $x_1$ and $x_2$ respectively. Then 
there exists a unique good lift $\widetilde{\phi}$ of $\phi$ such that $\widetilde{\phi}([\widetilde{x_2}, \widetilde{x_2}+1)) = \{\widetilde{x_1}\}$, 
and this lift clearly satisfies Property (4) of Proposition \ref{MainProposition}. This shows that (iii) implies (i) and finishes the proof.
\end{proof}
We have thus established the equivalence of the conditions (i)-(v) in Theorem \ref{ThmEquivalence}. Together with Lemma \ref{uppersemi} this finishes the proof of Theorem \ref{ThmEquivalence}.
 %



\subsection{Semi-conjugacy and monotone equivalence} Let us say that a circle action $\rho_1\colon \G\to H$ is \emph{left-equivalent} to another circle action $\rho_2\colon \G\to H$ if 
 $\rho_1$ is left-semi-conjugate to $\rho_2$ via a  continuous non-decreasing degree one map
$\varphi\colon \G\to H$ of Hopf-Brouwer degree 1, and recall from the introduction that  \emph{monotone equivalence} is defined as the equivalence relation generated by left-equivalence. 
This subsection is devoted to the proof of Theorem~\ref{Cal}, which states that monotone equivalence is equivalent to semi-conjugacy in the sense of the present note. One direction is immediate from what we have proved so far:
\begin{prop}\label{easyimplication}
 Suppose that $\rho_1,\rho_2\colon \G\to H$ are monotone equivalent circle actions. Then $\rho_1$ and $\rho_2$ are semi-conjugate.
\end{prop}
\begin{proof}
We may reduce to the case when $\rho_1$ is left-equivalent to $\rho_2$. In this case, $\rho_1$ is left-semi-conjugate to $\rho_2$ via a non-constant map, so
 the implication (2)$\Rightarrow$(6) of Propostion~\ref{MainProposition} and Theorem~\ref{ThmEquivalence} imply that $\rho_1$ is semi-conjugate to $\rho_2$. 
\end{proof}
Concerning the converse implication we recall the following classical trichotomy for circle actions (see e.g. \cite{Ghys2} for a detailed discussion and proof).
\begin{lemma} Let $\rho\colon\G\to H$ be a circle action. Then exactly one of the following three possibilities occurs:
\begin{enumerate}
 \item $\rho(\G)$ has a finite orbit.
 \item $\rho$ is minimal, i.e. every $\rho(\G)$-orbit is dense.
 \item there exists a unique $\rho(\G)$-invariant infinite compact proper subset $K\subsetneq S^1$ (called the \emph{exceptional minimal set of $\rho(\G)$}) such that $K$ is contained in the closure of any orbit of $\rho(\G)$.
\end{enumerate}
In case (3),  $K$ is homeomorphic to a Cantor set.\qed
\end{lemma}
From this we deduce:
\begin{prop}\label{nofinite}
Let $\rho\colon \G\to H$ be a circle action without finite orbits. Then $\rho$ is monotone equivalent to a minimal action. 
\end{prop}
\begin{proof} Since (1) is excluded by the assumption and the conclusion holds trivially in case (2), we may assume that $\rho$ satisfies (3) of the above trichotomy. Thus
let $K\subsetneq S^1$ be the minimal exceptional set of $\rho(\G)$. We have $S^1\setminus K=\bigcup_{i\in \mathbb{N}} U_i$,
 where the $U_i$'s are pairwise disjoint non-empty open subsets of $S^1$ homeomorphic to open intervals. Define an equivalence relation $\sim$ on $S^1$ by declaring $x\sim y$ if and only if there exists $i\in\mathbb{N}$ such that 
 $
 \{x,y\}\subset \overline{U_i},
$ 
 and let $\varphi: S^1 \to X := S^1/\sim$ denote the quotient map. Since $X$ is obtained from $S^1$ by collapsing intervals, it is homeomorphic to $S^1$. Moreover, the map $\varphi$ is a continuous, non-decreasing degree one map of Hopf-Brouwer degree $1$  (any of its lifts to $\R$ is just a devil's staircase). 
 
Now let $\gamma\in\G$. Since $K$ is $\rho(\gamma)$-invariant, the element $\rho(\gamma)$ permutes the intervals $U_i$ and thus descends to an orientation-preserving homeomorphism of $X$. We thus obtain a homomorphism $\rho'\colon \G \to {\rm Homeo}^+(X)$ such that for all $\gamma \in \G$, 
 \[\rho'(\gamma)\varphi=\varphi\rho(\gamma),\]
 and it remains to show only that $X$ is minimal under $\rho'(\Gamma)$. However, this follows from the observation that since $K$ is contained in the closure of any $\rho(\G)$-orbit, 
 the set $S^1=\varphi(K)$ is contained in the closure of any $\rho'(\G)$-orbit.
\end{proof}
The proof of Theorem~\ref{Cal} in the case where every orbit of $\rho_1(\G)$ and $\rho_2(\G)$ is infinite is now immediate.
\begin{proof}[Proof of Theorem~\ref{Cal} if every orbit of $\rho_1(\G)$ and $\rho_2(\G)$ is infinite] 

Let us  assume that $\rho_1$ and $\rho_2$ are semi-conjugate and that every orbit of $\rho_1(\G)$ and $\rho_2(\G)$ is infinite.
By Proposition~\ref{nofinite}, the actions $\rho_i$ are monotone equivalent to minimal actions $\rho_i'$ for $i=1,2$. We have already proved in Proposition~\ref{easyimplication}
that monotone equivalence implies semi-conjugacy, so $\rho'_1$ is semi-conjugate to $\rho'_2$. On the other hand, we know from Proposition~\ref{dense:orbits} that semi-conjugate minimal actions
are conjugate, whence in particular monotone equivalent. Since monotone equivalence is an equivalence relation, this implies that $\rho_1$ and $\rho_2$ are monotone equivalent.
\end{proof}
It remains to deal with the case where $\rho_1$ or $\rho_2$ has a finite orbit. This is slightly more technical.
\begin{proof}[Proof of Theorem~\ref{Cal} in the presence of a finite orbit] Here we assume that $\rho_1$ and $\rho_2$ are semi-conjugate via a pair of non-decreasing degree one maps $\varphi,\varphi'\colon S^1\to S^1$ satisfying
\[
\rho_1(\gamma)\varphi=\varphi\rho_2(\gamma)\quad \text{and} \quad \rho_2(\gamma)\varphi'=\varphi'\rho_1(\gamma)
\]
for every $\gamma\in\G$, and that one of them, say $\rho_1$, has a finite orbit $\{x_1, \dots, x_k\}$. We may assume that $(x_1,\ldots,x_k)$ is positively oriented. Note that, since $\rho_1(\G)$ acts transitively on the $x_i$'s, it also acts transitively
on the connected components of $S^1\setminus \{x_1,\ldots,x_k\}$. As a consequence, every orbit of $\rho_1(\G)$ must contain at least $k$ points. In particular, if we set $y_i=\varphi(x_i)$ for every $i=1,\ldots,k$, then $\{y_1,\ldots,y_k\}$ is a finite $\rho_2(\G)$-orbit, and $\varphi'(\{y_1,\ldots,y_k\})$ is a finite $\rho_1(\G)$-orbit, hence has to have at least $k$ points. This implies
that the $y_i$'s are pairwise distinct and that $(y_1,\ldots,y_k)$ is a positively-oriented $\rho_2(\G)$-invariant $k$-tuple. 

For every $\gamma \in \Gamma$ the homeomorphism $\rho_1(\gamma)$ induces a cyclic permutation of $(x_1, \dots, x_k)$, hence there exists $j(\gamma) \in \Z/k\Z$ such that 
\[
\rho_1(\gamma)x_i=x_{i+j(\gamma)},
\]
where addition of indices is always understood in $\Z/k\Z$. We can now compute the rotation number of $\rho_1(\gamma)$ using the orbit $\{x_1, \dots, x_k\}$; we then obtain
\[
R(\rho_1(\gamma)) = [j(\gamma)/k] \in \R/\Z.
\]
Note that the cyclic permutation induced by $\rho_1(\gamma)$ on $(x_1,\ldots,x_k)$ is completely determined by $R(\rho_1(\gamma))$. However, since the restrictions of $\rho_1$ and $\rho_2$ to the cyclic subgroup generated by $\gamma$ are semi-conjugate, it follows from Corollary~\ref{Poincare} that the rotation numbers of $\rho_1(\gamma)$ and $\rho_2(\gamma)$ coincide. We deduce that $\rho_1(\gamma)$ induces the same cyclic permutation on $(x_1, \dots, x_k)$ as $\rho_2(\gamma)$ on $(y_1, \dots, y_k)$. This information is enough to construct a circle action $\rho_3$ ``containing'' both $\rho_1$ and $\rho_2$ as follows.

Let us first assume that $k \geq 2$. Given two distinct points points $a, b \in S^1$ we define the open interval $(a,b)$ as
\[
(a,b):= \{z \in S^1\mid (a,z,b)\text{ positively oriented}\}.
\]
For every $i=1,\ldots,k$ we define $U_i := (x_i, x_{i+1})$ and $V_i:= (y_i, y_{i+1})$ and denote by $\overline{U_i}$ and $\overline{V_i}$ the closures of $U_i$ and $V_i$ in $S^1$ respectively. By the assumption $k \geq 2$ these are homeomorphic to closed intervals. We then define $X$ as the quotient space obtained from the disjoint union of the $\overline{U_i}$ and the $\overline{V_i}$ obtained by identifying the right endpoint $x_{k+1} \in \overline{U_{k}}$ with the left endpoint $y_{k} \in \overline{V_{k}}$ and the right endpoint $y_{k+1}\in \overline{V_{k}}$ with the left endpoint $x_{k+1} \in \overline{U_{k+1}}$. In the case $k=1$ we instead define $X$ by cutting $S^1$ at the respective fixed points $x_1$ and $y_1$ and glueing the resulting two open intervals $U_1$ and $V_1$ along a $0$-sphere. Either way we obtain a circle $X$ which contains $U_1, V_1, \dots, U_k, V_k$ in this exact cyclic order, and such that the complement of these open sets is a finite set of points. We now define $\rho_3: \Gamma \to {\rm Homeo}^+(X)$ as follows: Given $\gamma \in \G$ we define
\[
\rho_3(\gamma): \bigcup_{i=1}^k U_i \cup V_i \to \bigcup_{i=1}^k U_i \cup V_i, \quad \rho_3(\gamma)(x) = \left\{\begin{array}{rl}\rho_1(\gamma)(x), & x \in \bigcup U_i,\\ \rho_2(\gamma)(x), & x \in \bigcup V_i.\end{array}\right.
\]
Since $\rho_1(\gamma)$ induces the same permutation on the $x_i$ as $\rho_2(\gamma)$ on the $y_i$, it follows that $\rho_3(\gamma)$ extends uniquely to an orientation-preserving homeomorphism of $X \cong S^1$.

It remains to show only that $\rho_1$ and $\rho_2$ are left-equivalent to $\rho_3$, but this is obvious: Concerning $\rho_1$ we define a continuous non-decreasing map $\varphi: X \to S^1$ by contracting each of the intervals $\overline{V_i}$ to a point. Then $\varphi$ has Hopf-Brouwer degree $1$ and, by construction, $
\rho_1(\gamma)\varphi=\varphi\rho_3(\gamma)
$ holds for all $\gamma \in \Gamma$. Similarly, the left-equivalence from $\rho_2$ to $\rho_3$ is obtained by collapsing the $\overline{U_i}$. 
\end{proof}

\newpage
\appendix
\section{The action of the double cover of $H$ on the circle}\label{Section double cover}
Consider the circle $S^1$ and its double cover $X$ which, somewhat confusingly, is again homeomorphic to a circle. 
We denote by $\overline{H}$ the group of those homeomorphisms of $X$ which map antipodal points to antipodal points. The action of $\overline{H}$ on $X$ then factors through an action of $S^1$ and thus gives rise to a surjective homomorphism \[p:\overline{H}\rightarrow H := \omeo,\] 
which exhibits $\overline{H}$ as the unique double cover of $H$. Since $X \cong S^1$, the group $\overline{H}$ can also be seen as a subgroup of ${\rm Homeo}(X) \cong H$, but this is not the point of view we are going to take.

From now on we will denote the double covering of the circle simply by $S^1$, with the understanding that the action of $\overline{H}$ on $S^1$ is the one described above. This action is actually important in many applications, since it contains the action of ${\rm SL}_2(\R)$ on the circle obtained by letting ${\rm SL}_2(\R)$ act on $\R^2 \setminus\{0\}$ via the standard action and identifying $S^1$ with $(\R^2\setminus\{0\})/\R_{>0}$. 
This action in turn is a particular instance of the action of ${\rm SL}_n(\R)$ on  $S^{n-1} \cong (\R^n\setminus\{0\})/\R_{>0}$ for $n \geq 2$, and these generalizations play an important role concerning higher Euler classes.

The aim of this appendix is twofold: On the one hand, we describe all  homogeneous cocycles obtained as $\overline{H}$-invariant functions $(S^1)^3\rightarrow \mathbb{Z}$ and relate them to the cohomology class $p^*(\eu_b)\in H^2_b(\overline{H},\mathbb{Z})$. On the other hand, we establish a fixed point theorem (Theorem \ref{thm: appendix}) which is stronger than its analogue for $H$ (Corollary \ref{zeroprop}) 
since in this case a fixed point is not only equivalent to the vanishing of the pullback of the bounded Euler class, but further to the vanishing of the pullback of a particular cocycle. 

\subsection*{Non-degenerate homogeneous cochains}
For every point $x\in S^1$, we denote by $\overline{x}$ its antipodal point. We say that an $H$-orbit in $(S^1)^k$ is \emph{degenerate} if it contains a point of the form $(\dots, x, \dots, x, \dots)$ or of the form $(\dots, x, \dots, \bar x, \dots)$. Given $n \in \mathbb N$ let us denote by $(S^1)^{[n]} \subset (S^1)^n$ the union of all non-degenerate $\overline{H}$-orbits in $(S^1)^n$.
 We refer to an $\overline{H}$-invariant function $f: (S^1)^{[n+1]} \to \Z$ as a \emph{non-degenerate homogeneous $n$-cochain}. Note that if $(x_0, \dots, x_n) \in (S^1)^{[n+1]}$, then $(x_0, \dots, \widehat{x_i}, \dots, x_n) \in (S^1)^{[n]}$ for all $i = 0, \dots, n$, and hence the homogeneous differential defines a map
\[
\delta: {\rm Map}((S^1)^{[n]}, \Z)^{\overline{H}} \to {\rm Map}((S^1)^{[n+1]}, \Z)^{\overline{H}}
\]
for every $n$. We refer to elements in the kernel respectively image of this map as \emph{non-degenerate homogeneous cocycles}, respectively \emph{non-degenerate homogeneous coboundaries}. 

Every cochain $f \in {\rm Map}((S^1)^{n+1}, \Z)^{\overline{H}}$ restricts to a non-degenerate homogeneous cochain on $(S^1)^{[n+1]}$ and this restriction defines a chain map
\[
{\rm res}: ({\rm Map}((S^1)^{n+1}, \Z)^{\overline{H}}, \delta) \to ({\rm Map}((S^1)^{[n+1]}, \Z)^{\overline{H}}, \delta), \quad f \mapsto f|_{S^{[n+1]}}.
\]
\begin{lemma} The map ${\rm res}$ induces an isomorphism on the level of cohomology. In particular,
\begin{equation}\label{BuMoIso}
H_b^\bullet(\overline{H}\curvearrowright S^1) = H^\bullet(\overline{H}\curvearrowright S^1) \cong H^\bullet({\rm Map}((S^1)^{[n+1]}, \Z)^{\overline{H}}, \delta).
\end{equation}
\end{lemma}
\begin{proof} Since for every $n$ there will always be finitely many (non-degenerate) $\overline{H}$-orbits, it is immediate that $H_b^\bullet(\overline{H}\curvearrowright S^1) = H^\bullet(\overline{H}\curvearrowright S^1)$.

Following \cite{BucherMonod} we construct an extension map
\[
{\rm ext}: ({\rm Map}((S^1)^{[n+1]}, \Z)^{\overline{H}}, \delta) \to  ({\rm Map}((S^1)^{n+1}, \Z)^{\overline{H}}, \delta),\quad f\mapsto \widetilde{f}.
\]
which on the level of cohomology is an inverse to ${\rm res}$. Intuitively, in order to define $\widetilde{f}(x_0, \dots, x_n)$ for a degenerate $(n+1)$-tuple $(x_0, \dots, x_n)$ we want to move $x_n,\dots ,x_0$ (in this order) a very small amount in the positive direction to make the $(n+1)$-tuple non-degenerate, and then evaluate $f$ on the perturbed tuple. 
More precisely, if $x_n$ is equal to $x_i$ or $\overline{x_i}$ for $i\neq n$, replace $x_n$ by a point $x_n^+$ such that $(x_n,x_n^+,\overline{x_n})$ 
is positively oriented and no $x_i$ or $\overline{x_i}$, for $i\neq n$, lies in the positive direction between $x_n$ and $x_n^+$. Continue inductively for all $x_i$'s and set $\widetilde{f}(x_0,\dots,x_n):=f(x_0^+,\dots,x_n^+)$. As in \cite{BucherMonod} one then shows that ${\rm ext}$ is a chain map which is inverse to {\rm res} in cohomology.
\end{proof}
In view of the lemma we can represent every class in $H^\bullet(\overline{H}\curvearrowright S^1)$ by a non-degenerate homogeneous cocycle, and thus we will focus on non-degenerate homogeneous cocycles from now on.

\subsection*{ Non-degenerate orbits of $\overline{H}$ acting on $(S^1)^{n+1}$}
The classification of non-degenerate $\overline{H}$-orbits on $(S^1)^{n+1}$ for $n\leq 2$ is as follows.
\begin{itemize}
\item[(n=0)] The action of $\overline{H}$ on $S^1$ has exactly one non-degenerate orbit. 
\item[(n=1)] The action of $\overline{H}$ on two factors $(S^1)^2$ has two non-degenerate orbits: If $x,y\in S^1$ are chosen so that $(x,y,\overline{x})$ is a positively oriented triple, then we denote them by
\[\mathcal O^{(2)}_+ :=\overline{H}\cdot (x,y) \quad \mathrm{and} \quad  O^{(2)}_- := \overline{H}\cdot (y,x).\]
\item[(n=2)]The action of $\overline{H}$ on three factors $(S^1)^3$ has eight non-degenerate orbits. Choose distinct points $x_0,x_1,x_2\in S^1$ and suppose that $(x_0,x_1,x_2,\overline{x_0})$ is a strictly positively oriented quadruple. Then the orbits are given as follows. There are six non-degenerate orbits parametrized by the symmetric group $\mathrm{Sym}(3)$ over $\{0,1,2\}$ and given by
\[\mathcal O^{(3)}_\sigma := \overline{H}\cdot(x_{\sigma(0)},x_{\sigma(1)},x_{\sigma(2)}), \quad (\sigma \in \mathrm{Sym}(3)),\]
and there are two additional non-degenerate orbits given by
\[\mathcal O^{(3)}_+ := \overline{H}\cdot(x_0, x_2,\overline{x_1}) \quad\mathrm{and} \quad \mathcal O^{(3)}_-:=\overline{H}\cdot(x_0,\overline{x_1},x_2).\]
\end{itemize}

\subsection*{Non-degenerate homogeneous $2$-cocycles and non-degenerate homogeneous $2$-coboundaries} Denote by $p_2:S^1\rightarrow S^1$  the double cover given by identifying antipodal points. Then $p_2$ induces a map commuting with the map induced by $p:\overline{H}\rightarrow H$ 
 \[\begin{xy}
\xymatrix{
H^*(H\curvearrowright S^1; \Z)\ar[d]\ar[r]^{p_2^*}&H^*(\overline{H}\curvearrowright S^1; \Z) \ar[d] \\
H^*(H; \Z)\ar[r]^{p^*}&H^*(\overline{H}; \Z) 
}
\end{xy}\]
Specialising to degree $2$, we know that the left-hand side $H^2(H\curvearrowright S^1; \Z) $ is an infinite cyclic group generated by the class of the Euler cocycle $c$. Our goal now is to prove that the right-hand side $H^2(\overline{H}\curvearrowright S^1; \Z)$ is also infinite cyclic and to construct an explicit homogeneous cocycle
representing its generator.


To this end we first observe that a non-degenerate homogeneous $2$-cochain $f$ is given by the $8$ numbers
\[
f_\sigma := f|_{\mathcal O^{(3)}_\sigma}, \quad f_+ := f|_{\mathcal O^{(3)}_+}, \quad f_- := f|_{\mathcal O^{(3)}_-},
\]
where $\sigma \in \mathrm{Sym}(3)$. 
\begin{lemma} \label{appendix} A nondegenerate homogeneous $2$-cochain  $f$ is a cocycle if and only if
\begin{eqnarray*}
f_\mathrm{Id}=f_{(0\ 1 \ 2 )}=f_{(0 \ 2 \ 1)} =: f^+,\\
f_{(0\ 1)}=f_{(0\ 2)}=f_{(1\ 2)} =: f^-,\\
f^+ + f^-=f_++f_-,
\end{eqnarray*}
and $f$ is a coboundary if and only if there exists $w_{\pm} \in \R$ such that
\[
f_\mathrm{Id} = w_+, \quad f_{(0\ 1)} = w_-, \quad f_+ = 2w_+-w_-, \quad f_- = 2w_- - w_+.
\]
Furthermore there is an isomorphims $H^2(\overline{H}\curvearrowright S^1; \Z)\cong \mathbb{Z}$ given by sending $[f]\in H^2(\overline{H}\curvearrowright S^1; \Z) $ to $f_+-2f^++f^-\in \mathbb{Z}$.
\end{lemma}
\begin{proof} It  it is a matter of elementary case by case consideration of configurations of four points on the circle to show that the cocycle equation implies the $5$ identities above.  For example, let $x_0,x_1,x_2,x_3,\overline{x_1}$ be positively oriented points on $S^1$. Applying the cocycle relation $\delta f=0$ to $(x_1,x_2,x_3,x_0)$ and $(x_3,x_0,x_1,x_2)$ leads to the first two equalities defining $f^+$. Applying the relation to $(x_2,x_1,x_0,x_3)$ and $(x_0,x_3,x_2,x_1)$ gives the two next equalities defining $f^-$. Finally, $\delta f(x_3,x_0,x_2,\overline{x_1})=f_+-f^-+f_--f^+=0$.

Moreover, if $b$ is a $1$-cochain with $b|_{O^{(2)}_{\pm}}\equiv w_\pm$, then a routine computation yields
\[
(\delta b)_\mathrm{Id} = w_+, \quad (\delta b)_{(0\ 1)} = w_-, \quad (\delta b)_+ = 2w_+-w_-, \quad (\delta b)_- = 2w_- - w_+.
\]
It remains to show that there are no other relations satisfied by an arbitrary non-degenerate homogeneous $2$-cocycle. For this we observe that the quotient of the space of non-degenerate homogeneous $2$-cochains satisfying the $5$ identities above by the space of coboundaries is isomorphic to $\Z$ via the map $f \mapsto f_+ - 2f^+ + f^-$. If there were any other relations, then $H^2(\overline{H}\curvearrowright S^1; \Z)$ would be finite. However, it follows from the explicit formula for the orientation cocycle ${\rm Or}$ on $S^1$ that no multiple of $p_2^*{\rm Or}$ (see the table below for the values of $p_2^*{\rm Or}$) satisfies the coboundary equations above. Thus $p_2^*[{\rm Or}]$ is of infinite order in $H^2(\overline{H}\curvearrowright S^1; \Z)$, whence the latter cannot be finite.
\end{proof}
In particular, a non-degenerate $2$-cocycle $f$ is given by $4$ integers $f^+$, $f^-$, $f_+$, $f_-$ subject to the single relation $f^+ + f^-=f_++f_-$ (or equivalently by the $3$ integers $f^+$, $f^-$ and $f_-$).
\begin{defn} The \emph{Sullivan cocycle} $E_{\rm Sull}$ is the non-degenerate $2$-cocycle $f$ given by $f^+ = f^- = 0$, $f_+=1$, $f_- = -1$.
\end{defn}
This cocycle was found by Sullivan as an explicit representative for the Euler class of flat oriented $\R^2$-vector bundles. The following table compares the Sullivan cocycle with the pullback of the Euler cocycle via $p_2$ and also with the orientation cocycle on $S^1$ and the pullback of the orientation cocycle under $p_2$, and expresses all of these cocycles in terms of the $4$ integers $f^+$, $f^-$, $f_+$, $f_-$.

\vspace{.5cm}
\begin{tabular}{|c|c|c|c|c|c|}
\cline{2-6} 
\multicolumn{1}{c|}{} & {$f^{+}$} & $f^{-}$ & $f_{+}$ & $f_{-}$ & $H^{2}(\overline{H}\curvearrowright S^{1};\mathbb{Z})$\tabularnewline
\hline 
\hline
$E_{Sull}$ & $0$ & $0$ & $1$ & $-1$ & $\left[E_{Sull}\right]$\tabularnewline
\hline 
$p_{2}^{*}(c)$ & $1$ & $0$ & $0$ & $1$ & $-2\left[E_{Sull}\right]$\tabularnewline
\hline 
Or & $1$ & $-1$ & $1$ & $-1$ & $-2[E_{Sull}]$\tabularnewline
\hline 
$p_{2}^{*}(\mbox{Or})$ & $1$ & $-1$ & $-1$ & $1$ & $-4\left[E_{Sull}\right]$\tabularnewline
\hline 
$\delta b$ & $w_{+}$ & $w_{-}$ & $2w_{+}-w_{-}$ & $2w_{-}-w_{+}$ & $0$\tabularnewline
\hline
\end{tabular}
\vspace{.5cm}

In particular we see from the table and the isomorphism described in Lemma \ref{appendix} that the Sullivan class $[E_{Sull}]$ is a generator for $H^{2}(\overline{H}\curvearrowright S^{1};\mathbb{Z}) = H^{2}_b(\overline{H}\curvearrowright S^{1};\mathbb{Z})$. 

\subsection*{The geometric interpretation of the Sullivan cocycle}

Unravelling the definition and considering configurations of $3$ points on the circle case by case, we see that the Sullivan cocycle can be described geometrically as follows: it is nonzero on a non-degenerate triple $(x,y,z)$ if and only if the triple contains $0$ in the interior of its convex hull and in that case it is $+1$ or $-1$ depending on the orientation of the triple. 
This geometric definition generalizes to higher dimensions and leads to an ${\rm SL}_n(\R)$-invariant cocycle on the $(n-1)$-sphere for each $n \geq 2$.

One consequence of this description is that the Sullivan cocycle is not invariant under the full homeomorphism group of the circle, but only under its  subgroup $\overline{H}$.

Another useful consequence is that the Sullivan cocycle and its higher-dimensional analoga detect small subsets of spheres. Here a subset of a sphere is called \emph{small} if its spherical convex hull is not the whole sphere. In the case of $S^1$ a set $X \subset S^1$ is small if and only if it is contained in a half-open half-circle.

\begin{prop} \label{propAppendixSmall} Let  $X \subset S^1$ be any subset. Then $E_{Sull}$ vanishes on $X^3$ if and only if $X$ is small.
\end{prop}
\begin{proof} If $X\subset S^1$ is a small subset then no three points in $X$ ever contain $0$ in their convex hull, so that $E_{Sull}$ vanishes on $X^3$. 

Conversely, suppose that $E_{Sull}$ vanishes on $X^3$. View $X$ as a subset of $\mathbb{R}^2$ and consider its convex hull in $\mathbb{R}^2$. By Caratheodory's Theorem, if $0$ is contained in the convex hull of $X$, then there 
exist $x_0,x_1,x_2\in X$ such that $0$ belongs to the convex hull of $x_0,x_1,x_2$ and hence $E_{Sull}(x_0,x_1,x_2)\neq 0$, which is impossible. If $0$ is not on the boundary of the convex hull, 
then by Hahn-Banach there exists a hyperplane separating $0$ and the convex hull of $X$, so $X$ is in particular contained in the intersection of $S^1$ with the (appropriate) half plane delimited by the hyperplane. 
If $0$ is in the boundary of the convex hull, then by the supporting hyperplane theorem, there exists a hyperplane through $0$ so that the convex hull of $X$ is contained in one closed half space delimited by that hyperplane. We are almost done, except that we need to exclude the case that $X$ is contained in one closed half-circle, but is not contained in a half-open half-circle. 
Suppose that $x$ and $\overline{x}$ belong to $X$. Then $E_{Sull}(x,\overline{x},x)=E_{Sull}(x,\overline{x}^+,x^+)=1$, 
where the points $ \overline{x}^+,x^+\in S^1$ are very small perturbations of $\overline{x},x$ in the positive direction. \end{proof}
Note that the same proof holds also for the higher dimensional generalization of the Sullivan cocycle. 

\subsection*{The cohomology class $[E_{Sull}]$}
Given a basepoint $x \in S^1$ we obtain a cocycle $E_{Sull}^x: \overline{H}^3 \to \Z$ by pullback along the corresponding orbit map, i.e.
\[E_{Sull}^x(g_0, g_1, g_2) := E_{Sull}(g_0 x, g_1 x, g_2 x).\]
This cocycle determines a class in the group cohomology $H^2(\overline{H};\Z)$; since $E_{Sull}^x$ is bounded, it also determines a class in the bounded group cohomology $H^2_b(\overline{H};\Z)$. Recall from the table above that $-2\cdot [E_{\rm Sull}]=p_2^*c$, where $c$ denotes the Euler cocycle on $S^1$ and $p_2: S^1 \to S^1$ is the double covering.

Now the Euler class $\eu = [c_x] \in H^2(H, \mathbb{Z})$ 
corresponds to the central extension of $H$ given by the common universal covering group $\widetilde{H}$ of $H$ and $\overline{H}$, and thus it follows from the commuting diagram of central extensions
\[\begin{xy}
\xymatrix{
0\ar[r]&\Z\ar[r]^i\ar[d]^{\cdot 2}&\widetilde{H}\ar[d]^{\mathrm{Id}}\ar[r]^p&\overline{H}\ar[d]^{p} \ar[r]&\{e\}\\
0\ar[r]&\Z\ar[r]^i&\widetilde{H} \ar[r]^p&H \ar[r]&\{e\}
}
\end{xy}\]
that $[-E_{Sull}^x] \in H^2(\overline{H};\Z)$  corresponds to the central extension in the top row of the above diagram. By Lemma \ref{Lift} this yields the following interpretation of $[E_{Sull}]$ as an obstruction class: Given a group $\Gamma$, the $S^1$-action associated with a homomorphism $\rho: \Gamma \to \overline{H}$ lifts to an action of $\G$ on the real line if and only if $\rho^*[E^x_{Sull}]=0\in H^2(\Gamma,\Z)$ for some (hence any) $x\in S^1$.

\subsection*{The bounded cohomology class $[E_{Sull}]$}
We now turn to an interpretation of the bounded class defined by $E_{Sull}$. It turns out that the case of the bounded Sullivan cocycle in degree $2$ is very particular since the vanishing of the cohomology class is equivalent to the vanishing of the cocycle: 

\begin{prop} \label{propVanishingSullivanCocycle} Let $\Gamma$ be a group and $\rho: \Gamma \to \overline{H}$ be any homomorphism. Then $\rho^*[E^x_{Sull}]_b=0\in H^2_b(\Gamma,\Z)$ if and only if $\rho^*(E^x_{Sull})=0$ for any base point $x\in  S^1$.
\end{prop}

\begin{proof} The if-direction is trivial. For the only-if-direction, suppose that  $\rho^*E_{\rm Sull}^x = \delta b$  for some $x \in S^1$ and a $\Gamma$-invariant bounded cochain $b:\Gamma^2\rightarrow \Z$. We will show that $b \equiv 0$. Writing out the cocycle equation in a special case yields for all $\gamma \in \Gamma$,
\[\rho^*E_{\rm Sull}^x(e, \gamma, \gamma^2) = 2b(e, \gamma) - b(e, \gamma^2).\]
This implies in particular $|2b(e, \gamma) - b(e, \gamma^2)| \leq 1$, hence inductively
\begin{equation}\label{FunnyEstimate}
|2^k b(e, \gamma) - b(e, \gamma^{2^k})| \leq 2^k - 1.
\end{equation}
Since $b$ is bounded, we can choose $k$ sufficiently big so that $|b(e, \gamma^{2^k})|\leq 2^{k-1}$. Dividing \eqref{FunnyEstimate} by $2^k$ we obtain
$$ |b(e,\gamma)|\leq\frac{1}{2^k} |b(e,\gamma)^{2^k}|+1-\frac{1}{2^k}\leq 1+\frac{1}{2}-\frac{1}{2^k}<2.$$
Since $b$ takes integral values, it follows that it takes values in $\{-1, 0, 1\}$. Assume that $b(e, \gamma) =1$. Then \eqref{FunnyEstimate} yields
\[|2^k - b(e, \gamma^{2^k})| \leq 2^k-1,\]
hence $b(e, \gamma^{2^k}) = 1$. A similar argument in the negative case shows that for every $\gamma\in \Gamma$, either $b(e, \gamma) = 0$ or $0\neq b(e, \gamma) = b(e, \gamma^{2^k})$ for every $k>0$. Thus if $b(e, \gamma) \neq 0$ for some $\gamma$, then
\begin{eqnarray*}
E_{\rm Sull}(x, \rho(\gamma)x, \rho(\gamma)^2x) &=& 2b(e, \gamma) - b(e, \gamma^2) = b(e, \gamma) = b(e, \gamma^2)\\
&=& E_{\rm Sull}(x, \rho(\gamma)^2x, \rho(\gamma)^4x).
\end{eqnarray*}
This means that there exist $w,x,y,z \in S^1$ such that
\[E_{\rm Sull}(x, y,z) = E_{\rm Sull}(x, z,w) \neq 0.\]
By our extension of the Sullivan cocycle to degenerate orbits, we can without loss of generality suppose that both triples $(x,y,z)$ and $(x,z,w)$ are non-degenerate. Since their evaluations on the Sullivan cocycle agree both triples contain $0$ in the interior of their convex hull and have the same orientation. This is impossible. 
\end{proof}

For the Sullivan cocycle we now obtain the following stronger version of Corollary \ref{zeroprop}:

\begin{thm}\label{thm: appendix}
Let $\Gamma$ be a group, $\rho:  \Gamma \to \overline{H}$ a homomorphism. Then the following are equivalent:
\begin{enumerate}
\item $\rho^*[E_{Sull}^x]_b = 0 \in H^2_b(\Gamma;\Z)$;
\item $\rho$ lifts to a homomorphism $\widetilde{\rho}: \Gamma \to \widetilde{H}$ and $\widetilde{\rho}(\Gamma)$ has a fixed point in $\R$.
\item $\rho(\Gamma)$ fixes a point in $S^1$.
\item Every $\rho(\Gamma)$-orbit on $S^1$ is small.
\item There exists a small $\rho(\Gamma)$-orbit in $S^1$.
\item $\rho^*E_{Sull}^x = 0$ for every $x \in S^1$.
\item There exists $x \in S^1$ such that $\rho^*E_{Sull}^x = 0$.
\end{enumerate}
\end{thm}
\begin{proof} We summarize the shown implications in the following diagram:
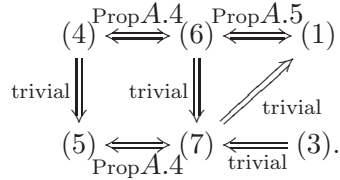
\[\begin{xy}
\xymatrix{
(4)\ar@{<=>}[r]^{\mathrm{Prop} \ref{propAppendixSmall}}\ar@{=>}[d]_{\mathrm{trivial}}&(6)\ar@{=>}[d]_{\mathrm{trivial}}\ar@{<=>}[r]^{\mathrm{Prop} \ref{propVanishingSullivanCocycle}}&(1)\\
(5)\ar@{<=>}[r]_{\mathrm{Prop} \ref{propAppendixSmall}}&(7)\ar@{=>}[ru]_{\mathrm{trivial}}&(3).\ar@{=>}[l]^{\mathrm{trivial}} 
}
\end{xy}\]
The remaining equivalences between (1), (2) and (3) admit the same proof as the equivalences  between (i), (ii) and (iii) in Corollary \ref{zeroprop}. 
\end{proof}

\bibliographystyle{amsalpha}
\bibliography{biblionote}
\end{document}